\documentclass{article}
\usepackage{amsmath}
\usepackage{amsthm}
\usepackage{graphics}
\usepackage{epsfig}

\newtheorem{theorem}{Theorem}[section]
\newtheorem{cor}[theorem]{Corollary}
\newtheorem{lem}[theorem]{Lemma}
\newtheorem{prop}[theorem]{Proposition}

\theoremstyle{definition}

\newtheorem{rem}[theorem]{Remark}

\newtheorem{example}[theorem]{Example}

\numberwithin{equation}{section}

\usepackage{amssymb}
\usepackage{amsfonts}
\usepackage{psfrag}
\usepackage{url}
\usepackage{mathtools}

\usepackage[makeroom]{cancel}

\usepackage{stmaryrd}

\newcommand{\N}{\mathbb{N}}
\newcommand{\R}{\mathbb{R}}

\newcommand{\Z}{\mathbb{Z}}
\newcommand{\Q}{\mathbb{Q}}

\newcommand{\cF}{\mathcal{F}}

\newcommand{\cP}{\mathcal{P}}
\newcommand{\cO}{\mathcal{O}}

\begin{document}

\title{The asymptotics of purely atomic\\one-dimensional L\'evy approximations}

 \author{Arno Berger\\Mathematical and Statistical
   Sciences\\University of Alberta\\Edmonton, Alberta, {\sc
     Canada}\\[2mm] and\\[2mm] Chuang Xu\\Department of Mathematical
   Sciences\\University of Copenhagen\\Copenhagen, {\sc Denmark}}

\maketitle

\begin{abstract}
\noindent
For arbitrary Borel probability measures on the real line, necessary
and sufficient conditions are presented that characterize best
purely atomic approximations relative to the classical L\'evy
probability metric, given any number of atoms, and allowing for
additional constraints regarding locations or weights of atoms. The
precise asymptotics (as the number of atoms goes to infinity) of the
approximation error is identified for the important special cases of
best uniform (i.e., all atoms having equal weight) and best
(unconstrained) approximations, respectively. When compared to similar
results known for other probability metrics, the
results for L\'evy approximations are more complete and
require fewer assumptions.
\end{abstract}

\noindent
\hspace*{8.3mm}{\small {\bf Keywords.} Best (uniform) approximation,
  L\'{e}vy probability metric, inverse function,\\
\hspace*{26.2mm} inverse measure, approximation error, asymptotic point
distribution.}

\noindent
\hspace*{8.3mm}{\small {\bf MSC2010.} 60B10, 60E15, 62E15.}

\section{Introduction}\label{sec1}

Let $\cP$ be the set of all Borel probability measures on the real
line, and denote the support of $\mu \in \cP$ by $\mbox{\rm supp}\,
\mu$. For each positive integer $n$, let $\cP_n^{\ast} = \{\mu \in \cP :
\# \, \mbox{\rm supp}\, \mu \le n \}$. Recall that $\cP$ endowed with the
topology of weak convergence is a Polish space that contains
$\cP_{\infty}^{\ast} := \bigcup_{n} \cP_n^{\ast} = \{ \mu \in \cP : \# \,
\mbox{\rm supp}\, \mu  < \infty \}$ as a dense subspace
\cite[Ch.11]{D}. Many different metrics (and metric-like quantities
\cite{GS}) on $\cP$ or parts thereof have been studied extensively,
as they play important roles in statistics and probability theory
\cite{R, RKSF}. Given a specific probability metric $d$ and $\mu \in \cP \setminus
\cP_{\infty}^{\ast}$, it is natural to ask whether there exists, for every $n$, a best $d$-approximation
$\delta_{\bullet}^{\bullet, n}$ of $\mu$ in $\cP_n^{\ast}$, i.e.,
$d(\mu , \delta_{\bullet}^{\bullet, n}) = \inf \{ d(\mu, \nu) : \nu
\in \cP_n^{\ast}\}$, perhaps with additional desirable properties such
as, e.g., all atoms having equal weight; see Section \ref{sec2} for
precise terminology and notation. Provided they exist, how can such
best $d$-approximations be characterized and found systematically? How
fast do they converge to $\mu$, i.e., at what rate does the
approximation error $d(\mu,\delta_{\bullet}^{\bullet, n})$ tend to $0$ as $n\to
\infty$? Questions like these, regarding the approximation in $\cP$ by
elements of $\cP_{\infty}^{\ast}$, continue to attract interest in a
wide variety of contexts; see, e.g., \cite{BA, BL, DSS, DV, GrayN, Gruber, PP}
and the many references therein.

Denoting the distribution function of $\mu \in \cP$ by $F_{\mu}$, i.e.,
$F_{\mu} (x) = \mu \bigl( ]-\! \infty, x]\bigr)$ for all $x\in \R$, recall the {\em
  Kantorovich\/} (or {\em Wasserstein}; cf.\ \cite[p.4]{BL} and \cite{GS}) {\em metric},
given by
\begin{equation}\label{eq101}
d_{\sf W} (\mu, \nu) = \int_{\R} |F_{\mu} (x) - F_{\nu} (x)| \, {\rm
  d}x = \int_{[0,1]} |F_{\mu}^{-1} (y) - F_{\nu}^{-1} (y)|\, {\rm d}y
\, ,
\end{equation}
where $F_{\mu}^{-1}$ is an inverse of $F_{\mu}$; see Section
\ref{sec2} for details. Note that strictly speaking $d_{\sf W}$ is not
defined on all of $\cP \times \cP$, but only on $\cP_1 \times \cP_1$,
with $\cP_1 = \bigl\{\mu \in \cP : \int_{\R} |x|\, {\rm d}\mu(x) < +
\infty \bigr\}\supset \cP_{\infty}^*$. The metric space $(\cP_1, d_{\sf W})$ is
complete and separable, though its metric topology is finer than the
subspace topology inherited from $\cP$. Due to its simplicity and
functional-analytic flavour, the metric $d_{\sf W}$ figures
prominently in many applied areas, e.g., image compression, signal
processing, mathematical finance, and optimal transport \cite{GrayN,
  PPP, RR, Vill, Z}. A vast literature exists addressing the
basic questions mentioned earlier relative to $d_{\sf W}$,
as well as many generalizations thereof, notably to
multi-dimensional settings \cite{BL, FP, GLP, GLP12, PP}.

Another important notion of distance, the {\em Prokhorov metric\/} is
given by
\begin{equation}\label{eq102}
d_{\sf P} (\mu, \nu) = \inf \bigl\{
y \in \R^+ : \mu(B) \le \nu (B^y) + y \enspace  \mbox{\rm for all Borel sets } B \subset \R
\bigr\} \quad \forall \mu, \nu \in \cP \, ,
\end{equation}
where $B^y = \{ x\in \R : \mbox{\rm dist} (x,B) < y\}$. Note that
$d_{\sf P}$ is defined on all of $\cP \times \cP$, unlike $d_{\sf W}$,
and metrizes precisely the topology of weak convergence
\cite{D, GS}. Also, $d_{\sf P} (\mu, \nu) \le 1$ for all $\mu, \nu \in
\cP$. A general theory of best $d_{\sf P}$-approximation in $\cP$ by elements
of $\cP_{\infty}^{\ast}$ has been initiated in \cite{GL1}, where the
authors observe that some aspects of the theory are ``more difficult [than the corresponding theory for
$d_{\sf W}$] \dots mainly due to the lack of suitable scaling
properties [of $d_{\sf P}$]''.

In a spirit similar to \cite{GL, XB}, the present article addresses
the approximation problem relative to the classical {\em L\'evy metric},
\begin{equation}\label{eq103}
d_1(\mu, \nu) = \inf \bigl\{
y\in \R^+ : F_{\mu-} (\, \cdot \, - y) - y \le F_{\nu} \le F_{\mu} (\,
\cdot \,  + y) + y
\bigr\} \quad \forall \mu, \nu \in \cP \, ,
\end{equation}
where $F_{\mu - } (x) = \lim_{\varepsilon \downarrow 0} F_{\mu} (x -
\varepsilon) = \mu \bigl( ]-\! \infty, x[\bigr)$. Note that $d_1 \le
1$, similarly to $d_{\sf P}$. The values of $d_{\sf W}$, $d_{\sf
P}$, and $d_1$ are not completely unrelated, since
$d_1 (\mu, \nu) \le d_{\sf P} (\mu, \nu) \le \sqrt{ d_{\sf W} (\mu,
  \nu) } $ for all $\mu, \nu \in \cP_1$; see, e.g., \cite{BX, D, GS}. When compared to $d_{\sf W}$ and $d_{\sf
  P}$, the metric $d_1$ is particularly attractive: On the one hand,
it is a {\em bona fide\/} metric \cite[p.100]{BHM} metrizing the
topology of weak convergence on all of $\cP$ (similar to $d_{\sf P}$,
but unlike $d_{\sf W}$). On the other hand, its definition
(\ref{eq103}) is considerably easier to work with than (\ref{eq102}). Although
computing $d_1$ for concrete problems may still be ``not
easy'' \cite[p.423]{GS} (cf.\ also \cite{T75}), especially when compared to (\ref{eq101}), the main (asymptotic)
results of this article suggest that nevertheless $d_1$ often is more benign
than both $d_{\sf W}$ and $d_{\sf P}$, in that fewer assumptions (or
no assumptions at all, as in Theorem \ref{thm450} and Proposition
\ref{prop454} below) are needed to draw
analogous or perhaps even stronger conclusions. With all technical
details deferred to later sections, this is illustrated here for two
familiar (absolutely continuous) distributions --- standard normal and $1$-Pareto.

Let $\mu$ be the standard normal distribution. By a celebrated
asymptotic result for best $d_{\sf W}$-approximations
\cite[Thm.6.2]{GL},
\begin{equation}\label{eq104}
\lim\nolimits_{n\to \infty} n d_{\sf W} (\mu,
\delta_{\bullet}^{\bullet, n}) = \sqrt{ \frac{\pi}{2} } = 1.253 \, ,
\end{equation}
whereas by \cite[Ex.5.2]{GL1},
\begin{equation}\label{eq105}
\lim\nolimits_{n\to \infty} \frac{n}{\sqrt{\log n}} \, d_{\sf P} (\mu,
\delta_{\bullet}^{\bullet, n}) = \sqrt{ 2 } \, .
\end{equation}
Note that (\ref{eq104}) yields the faster decay of the approximation (or {\em
  quantization\/}) error $d(\mu, \delta_{\bullet}^{\bullet, n})$,
whereas only (\ref{eq105}) involves a probability metric that actually
metrizes the topology of weak convergence. As it turns out, for the L\'evy metric
these two desirable properties can be achieved simultaneously: Theorem
\ref{thm450} below, one of the main results of this article, yields
\begin{equation}\label{eq106}
\lim\nolimits_{n\to \infty} n d_{1} (\mu,
\delta_{\bullet}^{\bullet, n}) =   -\sqrt{\frac{\pi}{2}} \, {\rm
  Li}_{1/2} \left( - \frac1{\sqrt{2\pi}}\right) = 0.3931  \, ,
\end{equation}
where ${\rm Li}_{1/2}$ denotes the polylogarithm of order
$\frac12$. An interesting variant of (\ref{eq104})--(\ref{eq106}) considers best
{\em uniform\/} approximations of $\mu\in \cP$, that is, best approximations
of $\mu$ by $\nu \in \cP_n^{\ast}$, subject to the additional
requirement that $n\nu (\{x\})$ is a (positive) integer for every
$x\in \, \mbox{\rm supp}\, \nu$. Best uniform (or, more generally,
best constrained) approximations have recently attracted considerable
interest, not least in view of potential applications in stochastic processes and
differential equations \cite{C, C18, GHMR, G18,  XThesis, XB}; they may also be viewed as deterministic analogues of
(random) empirical measures \cite{BL, DSS, FG}. With $\delta_{\bullet}^{u_n}$
denoting a best uniform $d$-approximation of $\mu$, trivially $d(\mu, \delta_{\bullet}^{\bullet, n}) \le d(\mu,
\delta_{\bullet}^{u_n})$. For $\mu$ being the standard normal
distribution, \cite[Ex.5.18]{XB} reports that
$$
d_{\sf W} (\mu, \delta_{\bullet}^{u_n}) = \cO \left(
\frac{\sqrt{\log n}}{n}
\right) \quad \mbox{\rm as } n \to \infty \, ,
$$
and this bound is sharp; cf.\ also \cite{C, GHMR}. Although the authors do
not know of any analogous result regarding best uniform $d_{\sf
  P}$-approximations, (\ref{eq105}) makes it clear that $d_{\sf P}
(\mu, \delta_{\bullet}^{u_n})$ is at least $\cO(n^{-1} \sqrt{\log n})$
as $n\to \infty$, if not larger. By contrast, Theorem \ref{thm304} below,
another main result of this article, simply yields
$$
\lim\nolimits_{n\to \infty} n d_1(\mu, \delta_{\bullet}^{u_n}) =
\frac12 \, ,
$$
which represents a faster and more precise rate than its $d_{\sf W}$-
and $d_{\sf P}$-counterparts.

For a second illustrative example, let $\mu$ be the $1$-Pareto
distribution, i.e., $F_{\mu} (x) = 1 - x^{-1}$ for all $x\ge
1$. Since $\mu \not \in \cP_1$, clearly $\mu$ is not amenable to
$d_{\sf W}$-approximation, whereas \cite[Thm.5.2]{GL1} yields
$$
\lim\nolimits_{n\to \infty} \sqrt{n} d_{\sf P} (\mu,
\delta_{\bullet}^{\bullet ,n}) = \sqrt{2} \, .
$$
For the L\'evy metric, this article again provides
faster, more precise rates, namely
\begin{equation}\label{eq1om1}
n  d_1 ( \mu ,  \delta_{\bullet}^{u_n}) = \frac12  -
\frac18 \, n^{-2} + \cO (n^{-3}) \quad \mbox{\rm as } n \to \infty \, ,
\end{equation}
as well as 
\begin{equation}\label{eq1om2}
n d_1 ( \mu , \delta_{\bullet}^{\bullet , n}) = \frac{\pi}{8}  +
\frac{\pi^2 (6-\pi)}{3\cdot 2^{10}} \, n^{-2} + \cO (n^{-3}) \quad \mbox{\rm as } n \to \infty \, .
\end{equation}
Thus the results of this article make the case that although the L\'evy
metric $d_1$, unlike $d_{\sf  W}$ and $d_{\sf P}$, does not extend to higher dimensions in a
straightforward way, its usage for one-dimensional probabilities often
leads to simpler and stronger results.

This article is organized as follows. Section
\ref{sec2} first introduces all required terminology and notation,
and then reviews basic facts pertaining to approximations in $\cP$ relative to the
L\'evy metric. Utilizing the latter, Sections \ref{sec3} and
\ref{sec4} specifically study best uniform and best (unconstrained)
approximations, respectively, and in particular the asymptotics of the
approximation error as $n\to \infty$. Also, under a mild assumption
the atoms of (asymptotically) best
approximations conform to an asymptotic point distribution, as shown
by Theorem \ref{thm470} below.

\section{L\'evy probability metrics}\label{sec2}

This section reviews basic facts regarding the approximation in $\cP$
by measures with finite support, relative to the L\'{e}vy probability
metric(s). The stated results are straightforward extensions of
\cite{BX, XThesis}, and the reader is referred to these references for
further details and elementary proofs. 
The following, mostly standard notations are used throughout. The sets
of all positive integers, non-negative integers, integers, positive
real numbers, and real
numbers are denoted $\N$, $\N_0$, $\Z$, $\R^+$, and $\R$,
respectively. Numerical values of real numbers are displayed to four
correct significant decimal digits. For every $x\in \R$ and non-empty $A\subset \R$,
$\mbox{\rm dist} (x,A) = \inf_{a\in A}|x-a|$, $\mbox{\rm diam}\, A =
\sup_{a,b\in A}|a-b|$, and ${\bf 1}_A$ is the indicator function of
$A$. The cardinality of $A$ is $\# A$. If the domain of a function $f$ contains $A$ then $f(A) = \{f(a)
: a\in A\}$. Lebesgue measure on the real line is denoted $\lambda$.

Since non-decreasing functions play a crucial role in what follows, first a few basic properties of such functions are
recorded. Throughout, denote by $\overline {\R} = \R \cup \{-\infty,
+\infty\}$ the extended real line with its usual order and
topology, and by $\cF$ the family of all functions $f:\R \to
\overline{\R}$ that are non-decreasing and right-continuous. Given
$f\in \cF$, let $f(\pm \infty) = \lim_{x\to \pm \infty} f(x)\in
\overline{\R}$, and for every $x\in \overline{\R}$ let $f_-(x) =
\lim_{\varepsilon \downarrow 0} f(x- \varepsilon)$. Note that $f_-(x) \le f(x)\le f_- (y)$ whenever $x<y$; in particular, $f_-(x)=f(x)$
if and only if $f$ is continuous at $x$. With every $f\in \cF$
associate its (upper) {\em inverse function\/} $f^{-1}:\R \to
\overline{\R}$ given by
$$
f^{-1} (x) = \sup \bigl\{ y\in \R : f(y) \le x \bigr\} \quad \forall x \in \R \, ;
$$
here and throughout the convention $\sup \varnothing = - \infty$ (and $\inf
\varnothing = +\infty$) is adhered to. Importantly, $\cF$ is closed
under inversion and composition.

\begin{prop}\label{prop201}
Let $f, g\in \cF$. Then $f^{-1}\circ g\in \cF$, and $(f^{-1})^{-1} = f$.
\end{prop}

\noindent
Given $f,g\in \cF$ and $\epsilon  >0 $, let 
$$
d_{\epsilon } (f,g) = \inf  \bigl\{y \in \R^+ : f_-( \, \cdot \, - y/
\epsilon  ) - y \le g \le
f(\, \cdot \, + y/ \epsilon ) + y\bigr\} \, \in \, [0, +\infty] \, .
$$
Motivated for $\epsilon  =1$ by (\ref{eq103}), this definition enables a unified treatment
of all $\epsilon $-L\'evy probability metrics later in this section. It is
readily checked that $d_{\epsilon }$ indeed satisfies the axioms of a metric on $\cF$,
except that $d_{\epsilon } (f,g)$ may equal $+\infty$. Also, $d_{\epsilon }$ is compatible with inversion.

\begin{prop}\label{prop202}
Let $f,g\in \cF$ and $\epsilon >0$. Then $d_{\epsilon}( f^{-1}, g^{-1})
= \epsilon  d_{1/\epsilon } (f,g)$.
\end{prop}

\noindent
Given $f,g\in \cF$, note that $\epsilon  \mapsto d_{\epsilon } (f,g)$
is non-decreasing and continuous on
$\R^+$. Consequently, the limits of $d_{\epsilon }
(f,g)$ exist as $\epsilon  \to 0$ or $\epsilon  \to +\infty$. For instance, if
$f,g\in \cF$ are bounded then simply
$$
\lim\nolimits_{\epsilon  \to 0} d_{\epsilon } (f,g) = \limsup\nolimits_{|x|\to
  +\infty} |f(x) - g(x)| = \max \bigl\{|f(-\infty) - g(-\infty)|,
|f(+\infty) - g(+\infty)| \bigr\} \, ,
$$
but also
$$
\lim\nolimits_{\epsilon  \to +\infty} d_{\epsilon } (f,g) = \sup\nolimits_{x\in
  \R} |f(x) - g(x)| = \|f-g\|_{\infty} \, ;
$$
here, as usual, $\|h\|_{\infty}  = \mbox{\rm ess sup}\, |h| = \inf \bigl \{y\in \R^+ : \lambda
(\{ |h|\ge y\}) =
0\bigr\}$ for every measurable function $h:\R \to
\overline{\R}$.

Given $f\in \cF$, let $I \subset \overline{\R}$ be any interval with
the property that
\begin{equation}\label{eq203}
f_- (\sup I - x), - f(\inf I + x) < + \infty \quad \mbox{\rm for some
} x \in \R \, ,
\end{equation}
and consider the auxiliary function $\ell_{f,I}: \R \to
\R$, introduced in \cite{BX}, with
$$
\ell_{f,I} (x) = \inf \bigl\{y\in \R^+ : f_- (\sup I - y) - y \le x \le
f(\inf I + y) + y \bigr\} \quad \forall x \in \R \, ;
$$
also, let $\ell_{f,I}^* = \inf \bigl\{y\in \R^+ : f_- (\sup I - y) - y \le
f(\inf I + y) + y\bigr\}$.
For any sequence $(I_k)_{k\in \N}$ of intervals in $\overline{\R}$, write $\lim_{k\to \infty} I_k = I$ if $\lim_{k\to \infty} \inf I_k =
\inf I$ and $\lim_{k\to \infty} \sup I_k  = \sup I$.

\begin{prop}\label{prop204}
Let $f\in \cF$, and let $I\subset \overline{\R}$ be an interval satisfying
{\rm (\ref{eq203})}.
\begin{enumerate}
\item The function $\ell_{f,I}$ is Lipschitz continuous and
  non-negative;
\item $\ell_{f,I} (x)\ge \ell_{f,I}^* \ge 0$ for all $x\in \R$;
\item $\ell_{f,I}^* \le \frac12 \lambda(I)$, and $\ell_{f,I}^* = 0$ if and only if $f_-(\sup I) \le f(\inf I)$;
\item If $(I_k)_{k\in \N}$ is a sequence of intervals in
  $\overline{\R}$ with $\lim_{k\to \infty} I_k = I$, then $I_k$ satisfies {\rm
    (\ref{eq203})} for all sufficiently large $k$, and
$$
\lim\nolimits_{k\to \infty} \ell_{f,I_k}^*
= \ell_{f,I}^*  \quad \mbox{as well as} \quad
\lim\nolimits_{k\to \infty} \ell_{f, I_k} (x) = \ell_{f,I} (x) \quad
\forall x \in \R\, .
$$
\end{enumerate}
\end{prop}

\begin{rem}\label{rem205}
(i) If $f = F_{\mu}$ (respectively, $f = F_{\mu}^{-1}$) for some $\mu \in
\cP$ then every (respectively, every bounded) interval $I\subset \overline{\R}$
satisfies (\ref{eq203}). Given $f\in \cF$, note that $f = F_{\mu}$ for
some (necessarily unique) $\mu \in \cP$ if and only if $f(-\infty) =
0$ and $f(+\infty) = 1$; similarly, $f = F_{\mu}^{-1}$ for some $\mu
\in \cP$ if and only if $f_-(0)=-\infty$, $f(1)=+\infty$, and
$f\bigl( ]0,1[ \bigr)\subset \R$.

(ii) The function $\ell_{f,I}$ may not attain a minimum value, or when it
does, that minimum value may be larger than $\ell_{f,I}^*$. However, mild
additional assumptions guarantee that $\ell_{f,I} (x) =
\ell_{f,I}^*$ for some $x\in \R$; see \cite[Prop.3.3]{BX}.
\end{rem}

For every $\epsilon  >0$, consider the $\epsilon  $-{\em L\'evy metric\/} on $\cP$ given by
$$
d_{\epsilon } (\mu, \nu) = d_{\epsilon } (F_{\mu}, F_{\nu}) \quad \forall \mu, \nu \in \cP
\, .
$$
The metric $d_{\epsilon }$ is complete, separable,
and induces the topology of weak convergence. (For an authoritative account on the family $(d_{\epsilon })_{\epsilon
  >0 }$ the
reader may want to consult \cite[Sec.4.2]{RKSF}; see also \cite{T75}.) Note that $\epsilon
\mapsto d_{\epsilon} (\mu, \nu)$ is non-decreasing with $\lim_{\epsilon 
\to 0} d_{\epsilon } (\mu, \nu) = 0$, whereas
$$
\lim\nolimits_{\epsilon  \to +\infty} d_{\epsilon } (\mu, \nu) = \| F_{\mu} - F_{\nu}\|_{\infty}
\quad \forall \mu , \nu \in \cP \, ,
$$
often referred to as the {\em uniform\/} or {\em Kolmogorov\/} metric,
yields a complete yet non-separable metric on $\cP$ and induces
a finer topology \cite[Sec.5]{BX}. For any $\mu \in \cP$ and $\epsilon  >0$, and with the dilation $T_{\epsilon } : x\mapsto \epsilon  x$, notice
the simple but useful identity
\begin{equation}\label{eq206}
d_{\epsilon } (\mu, \nu) = d_1 ( \mu \circ T_{\epsilon }^{-1}, \nu \circ T_{\epsilon }^{-1}) \quad \forall \mu, \nu
\in \cP \, .
\end{equation}
To study finitely supported (and hence purely atomic)
$d_{\epsilon }$-approximations of any $\mu \in \cP$, this article
employs the following notations: For every $n\in \N$, let $\Xi_n=\{x\in\R^n: x_{,1}\le\ldots\le x_{,n}\}$,
$\Pi_n=\{p\in\R^n: p_{,j}\ge0,\ \sum_{j=1}^np_{,j}=1\}$, and for
each $x\in \Xi_n$ and $p\in \Pi_n$ let $\delta_x^p=\sum_{j=1}^np_{,j}\delta_{x_{,j}}$. For convenience,
$x_{,0}:=-\infty$ and $x_{,n+1}:=+\infty$ for every $x\in\Xi_n,$ as
well as $P_{,i} :=\sum_{j=1}^ip_{,j}$ for $i=0,\ldots,n$ and every $p\in\Pi_n$; note that $P_{,0}=0$ and $P_{,n}=1.$
Henceforth, usage of the symbol $\delta_x^p$ tacitly assumes that $x\in\Xi_n$ and $p\in\Pi_n,$
for some $n\in\mathbb{N}$ either specified explicitly or else clear
from the context. Utilizing (\ref{eq206}) and \cite[Lem.3.4]{BX}, the
value of $d_{\epsilon } (\mu, \delta_x^p)$ allows for simple explicit expressions.

\begin{prop}\label{prop207}
Let $\mu\in \cP$, $\epsilon  > 0 $, and $n \in \N$. For every $x\in \Xi_n$
and $p\in \Pi_n$,
\begin{equation}\label{eq208}
d_{\epsilon } (\mu, \delta_x^p) = \epsilon  \max\nolimits_{j=0}^n
\ell_{F_{\mu}/\epsilon  , [x_{,j},
  x_{,j+1}]} (P_{, j}/\epsilon )  = \max\nolimits_{j=1}^n \ell_{\epsilon  
  F_{\mu}^{-1}, [P_{,j-1}, P_{,j}]} (\epsilon   x_{,j}) \, .
\end{equation}
\end{prop}

\noindent
For every $\mu\in \cP$, $\epsilon >0$, and $n\in \N$, (\ref{eq208})
suggests considering the following quantities: Given $x\in \Xi_n$, let 
$$
\ell^{\bullet}_x = \epsilon  \max \left\{
\ell_{F_{\mu}/\epsilon  , [-\infty, x_{,1}]}(0), \ell_{F_{\mu}/\epsilon  , [x_{,1},
  x_{,2}]}^*, \ldots,  \ell_{F_{\mu}/\epsilon  , [x_{,n-1},
  x_{,n}]}^*, \ell_{F_{\mu}/\epsilon , [x_{,n}, +\infty]}(1/\epsilon  )
\right \} \, ,
$$
and given $p\in \Pi_n$, let
$$
\ell_{\bullet}^p = \max\nolimits_{j=1}^n
\ell_{\epsilon  F_{\mu}^{-1}, [P_{,j-1}, P_{,j}]}^* \, .
$$
Notice that while $\ell^{\bullet}_x$ and $\ell_{\bullet}^p$ do depend
on $\mu, \epsilon $, and, implicitly, also $n$, in order to keep notations simple, this dependence is
not displayed explicitly. By Proposition \ref{prop204}(iv), $p\mapsto \ell_{\bullet} ^p$ is
continuous on $\Pi_n$, and hence
$$
\ell_{\bullet}^{\bullet,n } = \min \nolimits_{p\in \Pi_n} \ell_{\bullet} ^p
$$
is well-defined. (For a constructive alternative definition of
$\ell^{\bullet, n}_{\bullet}$, see \cite[Sec.3]{BX}.) The quantities $\ell^{\bullet}_x$, $\ell_{\bullet}^p$, and
$\ell_{\bullet}^{\bullet,n}$ control the minimization of $(x,p)\mapsto d_{\epsilon } (\mu,
\delta_x^p)$, with or without
constraints, in a sense made precise by Proposition \ref{prop209} below. To
formulate the result, call $\delta_x^p$ a {\em best
  $d_{\epsilon }$-approximation of $\mu\in \cP$, given $x\in \Xi_n$}  if
$$
d_{\epsilon } (\mu,\delta_x^p)\le d_{\epsilon }(\mu,\delta_x^q )\quad  \forall\
q\in\Pi_n\, .
$$
Similarly, call $\delta_x^p$ a  {\em best
  $d_{\epsilon }$-approximation of $\mu$, given $p\in \Pi_n$} if
$$
d_{\epsilon } (\mu,\delta_x^p)\le d_{\epsilon }(\mu,\delta_y^p)\quad  \forall\
y\in\Xi_n\, .
$$
Denote by $\delta_x^{\bullet}$ and $\delta_{\bullet}^p$ any best
$d_{\epsilon }$-approximation of $\mu$, given $x$ and $p$, respectively. Best
$d_{\epsilon }$-approximations, given $p=u_n := (n^{-1}, \ldots , n^{-1})$ are referred to
as best {\em uniform\/} $d_{\epsilon }$-approximations, and denoted
$\delta_{\bullet}^{u_n}$. Finally, $\delta_x^p$ is a {\em best
  $d_{\epsilon }$-approximation of\/} $\mu \in \cP$, denoted
$\delta_{\bullet}^{\bullet, n}$, if
$$
d_{\epsilon }(\mu,\delta_x^p)\le d_{\epsilon } (\mu,\delta_y^q)\quad  \forall\
y\in \Xi_n, q\in\Pi_n\, .
$$ 
Notice that usage of the symbols $\delta_x^{\bullet},$
$\delta^p_{\bullet},$ and $\delta_{\bullet}^{\bullet,n}$ always refers
to specific $\mu \in \cP$, $\epsilon >0 $, and $n\in \N$, all of which are usually
clear from the context.

\begin{prop}\label{prop209}
Let $\mu\in \cP$, $\epsilon  >0$, and $n\in \N$.
\begin{enumerate}
\item For every $x\in \Xi_n$, there exists a best $d_{\epsilon }$-approximation
  of $\mu$, given $x$, and $d_{\epsilon } (\mu, \delta_x^{\bullet}) =
  \ell^{\bullet}_x$. Moreover, $d_{\epsilon } (\mu, \delta_x^p) =
  \ell^{\bullet}_x$ with $p\in \Pi_n$ if and only if
\begin{equation}\label{eq210}
\epsilon  \ell_{F_{\mu}/\epsilon  , [x_{,j}, x_{, j+1}]} (P_{,j}/\epsilon ) \le \ell^{\bullet}_x \quad \forall j =0, \ldots , n \, .
\end{equation}
\item For every $p\in \Pi_n$, there exists a best $d_{\epsilon }$-approximation
  of $\mu$, given $p$, and $d_{\epsilon } (\mu, \delta_{\bullet}^p) =
  \ell_{\bullet} ^p$. Moreover, $d_{\epsilon } (\mu, \delta_x^p)
  =\ell_\bullet^p$ with $x\in \Xi_n$ if and only if
\begin{equation}\label{eq211}
\ell_{\epsilon  F_{\mu}^{-1}, [P_{,j-1}, P_{, j}]} (\epsilon  x_{,j}) \le  \ell_{\bullet}^p \quad \forall j =1, \ldots , n \, .
\end{equation}
\item There exists a best $d_{\epsilon }$-approximation of $\mu$, and $d_{\epsilon } (\mu,
  \delta_{\bullet}^{\bullet, n}) = \ell_{\bullet}^{\bullet,
    n}$. Moreover, $d_{\epsilon } (\mu, \delta_x^p) =
  \ell_{\bullet}^{\bullet, n} $ with $x\in \Xi_n$, $p\in \Pi_n$ if
  and only if {\rm (\ref{eq210})} and {\rm (\ref{eq211})} hold with
  $\ell^{\bullet, n}_{\bullet}$ instead of $\ell^{\bullet}_x$
  and $\ell_{\bullet} ^p$, respectively.
\end{enumerate}
\end{prop}

\noindent
The following two examples illustrate Proposition \ref{prop209}. Notice
that in either example the sequences $\bigl( d_{\epsilon } (\mu, \delta_{\bullet}^{u_n})
\bigr)$ and $\bigl( d_{\epsilon } (\mu, \delta_{\bullet}^{\bullet, n})
\bigr)$ both converge to $0$ at the same rate $(n^{-1})$. As demonstrated in Sections
\ref{sec3} and \ref{sec4} for best uniform and best
$d_{\epsilon }$-approximations, respectively, this rate is not specific to these
examples, but rather indicative of much more general mechanisms.

\begin{example}\label{ex212}
Consider the exponential distribution {\sf exp$(a)$} with parameter $a>0$, i.e., let
$F_{\mu} (x) = 1 - e^{-ax}$ for all $x\ge 0$. From Proposition
\ref{prop209} it is easily deduced that $\delta_x^{u_n}$ with $x\in \Xi_n$ is a best
uniform $d_{\epsilon }$-approximation of $\mu$ if and only if
$$
x_{,j} \in \left[
-\frac1{a} \log \left( 1 - \frac{j}{n} + \ell_{\bullet} ^{u_n} \right)
- \frac{\ell_{\bullet} ^{u_n}}{\epsilon } , -\frac1{a} \log\left(  1 -
  \frac{j-1}{n} - \ell_{\bullet}^{u_n} \right) + \frac{\ell_{\bullet} ^{u_n}}{\epsilon }
\right] \quad \forall j=1,\ldots , n \, ,
$$
with $\ell_{\bullet} ^{u_n} = d_{\epsilon } (\mu, \delta _{\bullet}^{u_n})$
being the unique solution of $n\ell (e^{2a\ell /\epsilon } + 1) =1$.
A straightforward analysis of the latter equation yields the
asymptotic equality
\begin{equation}\label{eq214}
n d_{\epsilon } (\mu, \delta_\bullet^{u_n}) = \frac12 -
\frac{a}{4\epsilon } \, n^{-1} +
\cO (n^{-2}) \quad \mbox{\rm as } n\to \infty \, .
\end{equation}
A best $d_{\epsilon }$-approximation of $\mu$ also exists, and in fact is
unique, with
$$
x_{,j} = -\frac1{a} \log \left( \frac{e^{2 a \ell_{\bullet}^{\bullet,
        n} (n-j)/\epsilon } - 1}{ e^{2a\ell_{\bullet}^{\bullet , n}n/\epsilon }
    -1} + \ell_{\bullet}^{\bullet, n}\right) -
\frac{\ell_{\bullet}^{\bullet , n}}{\epsilon }
\, , \quad P_{,j} =  \frac{1 - e^{-2a\ell_{\bullet}^{\bullet,
      n}j/\epsilon }}{ 1 - e^{-2a\ell_{\bullet}^{\bullet,
      n}n/\epsilon } }\quad \forall j = 1, \ldots , n \, ,
$$
where $\ell_{\bullet}^{\bullet , n}  = d_{\epsilon } (\mu, \delta
_{\bullet}^{\bullet, n})$ solves $\ell e^{2na \ell/\epsilon } = \ell +
\tanh (a\ell/\epsilon )$. Similarly to before, an analysis of this equation yields
\begin{equation}\label{eq2155}
n d_{\epsilon } (\mu, \delta_{\bullet}^{\bullet, n}) = c_1 - \frac{a^ 2
  c_1^2 }{6\epsilon  (a+\epsilon )} \, n^{-2}
+ \cO (n^{-4}) \quad \mbox{\rm as } n \to \infty \, ,
\end{equation}
with $c_1 = \frac12 \epsilon  \log (1+a/\epsilon ) /a < \frac12$. Note that
$d_{\epsilon } (\mu, \delta_{\bullet}^{\bullet, n}) <
d_{\epsilon } (\mu, \delta_{\bullet}^{u_n}) $ for every $n\ge 2$; see
also Figure \ref{fig2}.
\end{example}

\begin{figure}[ht]
\psfrag{tv1}[]{\small $1$}
\psfrag{tv05}[]{\small $\frac12$}
\psfrag{th0}[]{\small $0$}
\psfrag{th1}[]{\small $1$}
\psfrag{th2}[]{\small $2$}
\psfrag{th3}[]{\small $3$}
\psfrag{texp}[]{\small $\mu = {\sf exp}(1)$}
\psfrag{tu4}[r]{\small $4d_{1} (\mu , \delta_{\bullet}^{u_4}) = 0.4446\quad$}
\psfrag{tui}[r]{\small $\lim_{n\to \infty} n d_{1} (\mu ,
  \delta_{\bullet}^{u_n}) = \frac12 =  0.5000$}
\psfrag{tb4}[r]{\small $4d_{1} (\mu , \delta_{\bullet}^{\bullet, 4}) =
  0.3459\:\:$}
\psfrag{tbi}[r]{\small $\lim_{n\to \infty} n d_{1} (\mu ,
  \delta_{\bullet}^{\bullet, n}) =  \frac12 \log 2 =  0.3465$}
%
%
\begin{center}
\includegraphics{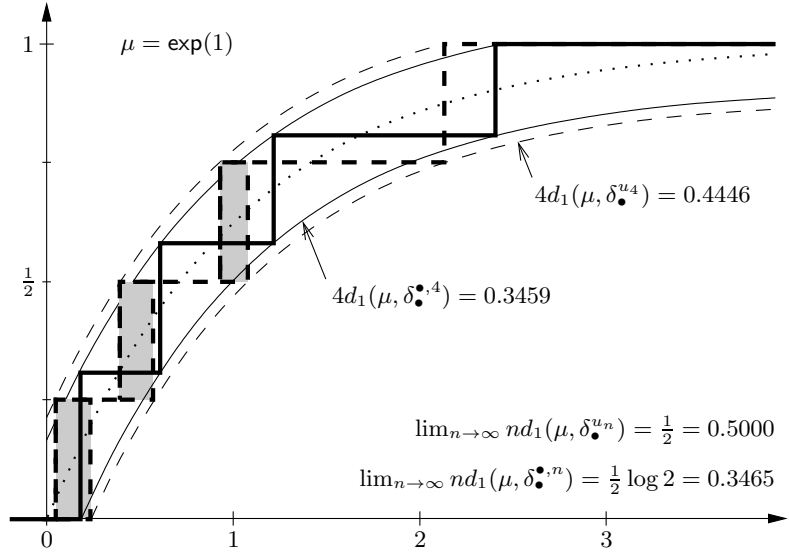}
\end{center}
\caption{Illustrating $d_1$-approximations of the standard exponential
distribution (dotted curve) with $n=4$ atoms: While the best approximation
$\delta_{\bullet}^{\bullet, 4}$ (solid
line) is unique, best uniform approximations $\delta_{\bullet}^{u_4}$ (broken lines) are not;
see Example \ref{ex212}.}\label{fig2}
\end{figure}

\begin{example}\label{ex216}
Fix $b>1$, and let $F_{\mu} (x) = \displaystyle \frac{\log x}{\log b}$
for all $1\le x \le b$. Usually referred to as {\em Benford's law\/} base
$b$, this distribution has many interesting properties; see, e.g.,
\cite{BH, BX} and the references therein. 
As in the previous example,
best uniform $d_{\epsilon }$-approximations of $\mu$ are non-unique, 
yet $\ell_{\bullet}^{u_n}$ is the unique solution of $b^{1-\ell} -
b^{1+\ell - 1/n} = 2\ell /\epsilon $, which in turn yields
\begin{equation}\label{eq217}
n d_{\epsilon } (\mu, \delta_\bullet^{u_n}) = c_2  -
\frac{c_2^2}{b\epsilon  } \, n^{-1} +
\cO (n^{-2}) \quad \mbox{\rm as } n\to \infty \, ,
\end{equation}
with $c_2 = \frac12 \epsilon  b \log b/ (1 + \epsilon  b\log b) <
\frac12$. Also similarly to Example \ref{ex212}, best
$d_{\epsilon }$-approximations of $\mu$ are unique, $\ell_{\bullet}^{\bullet,n}$
solves $b^{2n\ell} \bigl( \ell + \epsilon  \sinh (\ell \log b) \bigr) =
\ell + \epsilon  b \sinh (\ell \log b)$, and a straightforward
analysis yields
\begin{equation}\label{eq218}
n d_{\epsilon } (\mu, \delta_{\bullet}^{\bullet, n}) = c_3  + \frac{(b-1)
  c_2^2 c_3^2 b^{2c_3 - 2}}{3\epsilon  } \, n^{-2} + \cO (n^{-4}) \quad \mbox{\rm as } n \to \infty \, ,
\end{equation}
where $c_3  =\displaystyle \frac{\log (1+\epsilon  b\log b) - \log (1+\epsilon 
  \log b)}{2\log b} < c_2$; again $d_{\epsilon } (\mu,
\delta_{\bullet}^{\bullet, n}) < d_{\epsilon } (\mu,
\delta_{\bullet}^{u_n}) $ for every $n\ge 2$.
\end{example}

\section{Best uniform L\'evy approximations}\label{sec3}

This section provides a detailed asymptotic analysis of $d_{\epsilon } (\mu,
\delta_{\bullet}^{u_n})$ for any $\mu \in
\cP$. Notice the uniform bound $n d_{\epsilon } (\mu,
\delta_{\bullet}^{u_n}) \le \frac12$, due to Proposition \ref{prop204}. Thus $  d_{\epsilon } (\mu,
\delta_{\bullet}^{u_n})\to 0$ as $n\to \infty$ at an (upper) rate not slower
than $(n^{-1})$. Except for trivial cases, this rate is sharp.

\begin{lem}\label{lem301}
Let $\mu \in \cP$ and $\epsilon >0$. Then $\limsup_{n\to \infty} n d_{\epsilon } (\mu,
\delta_{\bullet}^{u_n}) >0$ unless $\mu = \delta_a$ for some $a\in \R$.
\end{lem}

\begin{proof}
Throughout the proofs of this section, write $g = F_{\mu}^{-1}$ for
convenience, and let $\omega_n = n d_{\epsilon} (\mu,
\delta^{u_n}_{\bullet})$ for all $n\in \N$, as well as $\omega^- =
\liminf_{n\to \infty} \omega_n$ and $\omega^+ = \limsup_{n\to \infty} \omega_n$.
Since $\delta_a \circ T_{\epsilon } ^{-1} = \delta_{\epsilon  a}$, by (\ref{eq206}) it
suffices to consider the case of $\epsilon  = 1$. Fix any
$0<x<y<1$. Assume that $\omega^+ = 0$, i.e., $\lim_{n\to \infty} \omega_n = 0$. Note that, for every $n\in \N$,
\begin{equation}\label{eq302}
g_- \left( \frac{j - 2\omega_n}{n}\right) - \frac{2\omega_n}{n} \le g
\left(
\frac{j -1 +  2\omega_n}{n}
\right) + \frac{2\omega_n}{n}  \quad \forall j = 1 , \ldots , n \, ,
\end{equation} 
by the definition of $\omega_n$. Also, observe that for all sufficiently large $n$,
$$
\frac{j + 2 \omega_{n+1}}{n+1} < \frac{j - 2\omega_n}{ n} \quad
\mbox{\rm and} \quad
\frac{j + 2\omega_n}{n} < \frac{j+1 - 2\omega_{n+1}}{n+1} \quad
\forall j = \lfloor nx \rfloor , \ldots , \lfloor ny \rfloor + 1 \, ,
$$
which, together with (\ref{eq302}), yields
\begin{align*}
g(y) - g(x) & \le g \left( \frac{\lfloor ny \rfloor + 1 +
    2\omega_n}{n}\right) - g_- \left( \frac{\lfloor nx \rfloor -
    2\omega_n}{n}\right) \\[2mm]
& = \sum\nolimits_{j = \lfloor nx \rfloor}^{\lfloor ny \rfloor + 1}
\left( g \left( \frac{j +
    2\omega_n}{n}\right) - g_- \left( \frac{j -
    2\omega_n}{n}\right) \right) + \\
& \qquad \qquad + \sum\nolimits_{j = \lfloor nx \rfloor +1}^{\lfloor
  ny \rfloor + 1} \left( g_- \left( \frac{j -
    2\omega_n}{n}\right) - g \left( \frac{j - 1 +
    2\omega_n}{n}\right)\right)  \\[2mm]
& \le \sum\nolimits_{j = \lfloor nx \rfloor}^{\lfloor ny \rfloor + 1}
\left( g_- \left( \frac{j + 1 -    2\omega_{n+1}}{n+1}\right) - g \left( \frac{j +    2\omega_{n+1}}{n+1}\right) \right)
+ \frac{4\omega_n}{n} (\lfloor ny \rfloor - \lfloor nx \rfloor + 1) \\[2mm]
& \le \frac{4\omega_{n+1}}{n+1} (\lfloor ny \rfloor - \lfloor nx \rfloor +
2) + \frac{4\omega_n}{n} (\lfloor ny \rfloor - \lfloor nx \rfloor +
1) \le 12 \omega_{n+1} + 8 \omega_n  \, . 
\end{align*}
Since $\lim_{n\to \infty} \omega_n = 0$ by assumption, and $0<x<y<1$ have been arbitrary, $g(0) = g_- (1)= a$ for some $a\in
\R$, that is, $\mu = \delta_a$.
\end{proof}

For the subsequent finer analysis, the following terminology is useful: For every $f\in
\cF$, let $G_f$ be the {\em growth set\/} of $f$, i.e., let
$$
G_f = \bigl\{ x\in \R : f(x- \varepsilon) < f (x+ \varepsilon) \: \forall
\varepsilon > 0\bigr\} \, .
$$
Note that $G_f$ is closed in $\R$, and $G_f \ne \varnothing$ unless
$f$ is constant. For example, $G_{F_{\mu}} = \mbox{\rm supp}\, \mu$
and $\{0,1\} \subset G_{F_{\mu}^{-1}} \subset [0,1]$ for every $\mu
\in \cP$. Also, $f(x)\in \R$ whenever $f^{-1} (-\infty) < x < f^{-1}
(+\infty)$. With $\lambda_f \bigl( ]-\! \infty , f^{-1}(-\infty )]\bigr) := \lambda_f
\bigl( [f^{-1} (+\infty),
+\infty[\bigr) := 0$ and
$$
\lambda_f \bigl( ]x,y] \bigr) := f(y) - f(x) \quad \forall f^{-1} (-\infty) < x
\le y < f^{-1} (+\infty) \, ,
$$
therefore, $\lambda_f$ is a $\sigma$-finite positive Borel measure
concentrated on $G_f$. For example, $\lambda_{{\rm id}_{\R}} = \lambda$,
and $\lambda_{F_{\mu}} = \mu$ for every $\mu
\in \cP$. Also, $\mu^{-1} := \lambda_{F_{\mu}^{-1}}$ is a positive Borel measure
supported on $G_{F_{\mu}^{-1}}\subset [0,1]$,
referred to as the {\em inverse measure\/} of $\mu$; see, e.g., \cite{BL, XB}. For convenience, write
$G_{F_{\mu}}$ and $G_{F_{\mu}^{-1}}$ simply as $G_{\mu}$ and
$G_{\mu^{-1}}$, respectively. Note that $\mu^{-1} (\R) = \mu^{-1}
\bigl( ]0,1[ \bigr)=
\mbox{\rm diam}\, G_{\mu}$, and hence
$\mu^{-1} = 0$ precisely if $\mu = \delta_a$ for some $a\in
\R$. When rephrased utilizing this terminology, Lemma \ref{lem301} has the following
corollary.

\begin{prop}\label{prop303}
For every $\mu \in \cP$ and $\epsilon >0$, the following are
equivalent: 
\begin{enumerate}
\item $\lim_{n\to \infty} n d_{\epsilon } (\mu, \delta_{\bullet}^{u_n})=0$;
\item $d_{\epsilon } (\mu, \delta_{\bullet}^{u_n})=0$ for every $n\in \N$;
\item $\mu = \delta_a$ for some $a\in \R$;
\item $\mu^{-1} = 0$.
\end{enumerate}
\end{prop}

\noindent
The first main result in this section asserts that $\bigl( n d_{\epsilon }
(\mu, \delta_{\bullet}^{u_n})\bigr)$ does converge, to an easily
determined limit, if $\mu^{-1}$ is absolutely continuous. 
The result is reminiscent of a theorem regarding best uniform $d_{\sf
  W}$-approximations \cite[Thm.5.15]{XB} (see also \cite{C, GHMR}), but unlike in that theorem,
no integrability assumption on ${\rm d} \mu^{-1}/{\rm d}\lambda$ is
needed, and the limit in question always is finite. When formulating the result, it is
helpful to use the function $\Omega : \R \to \R$ with
$$
\Omega  (x) =   \frac{ x}{2 +  2 |x|} \quad \forall x \in
\R \, .
$$
Plainly, $\Omega$ is an increasing $C^1$-function, with
$|\Omega(x)|\le \frac12 |x|$ for all $x\in \R$, and
$\Omega  (\pm \infty) = \pm \frac12$. While the appearance of $\Omega$
in the following theorem is a simple consequence of the bound
(\ref{eq306}), the reader may find it curious to note that $2\Omega$
plays a prominent role in the theory of random walks \cite{DE}.

\begin{theorem}\label{thm304}
Let $\mu\in \cP$ and $\epsilon >0$. If $\mu^{-1}$ is absolutely
continuous (w.r.t.\ $\lambda$) then
\begin{equation}\label{eq305}
\lim\nolimits_{n\to \infty} n d_{\epsilon } (\mu , \delta_{\bullet}^{u_n})
= \left\| \,  \Omega \!  \left( \epsilon \frac{{\rm d}\mu^{-1}}{{\rm d}\lambda}
  \right)\right\|_{\infty} \, .
\end{equation}
\end{theorem}

\begin{proof}
Since $\displaystyle \frac{{\rm d} (\mu \circ T_{\epsilon }^{-1})^{-1}}{{\rm d}\lambda} =
\epsilon   \frac{{\rm d} \mu^{-1}}{{\rm d} \lambda}$, it is enough to
prove (\ref{eq305}) for $\epsilon  = 1$. Using the same symbols as in
the proof of Lemma \ref{lem301}, for every
$n\in \N$ let 
$$
J_{n,j} = \left[ \frac{j-1 + \omega_n}{n} , \frac{j -
    \omega_n}{n}\right] \quad \forall j = 1 , \ldots , n \, .
$$
Note that $\omega_n < \frac12$ for every $n$ since $g$ is
continuous. Moreover, 
$$
g\left( \frac{j- \omega_n}{n}\right) - g \left( \frac{j-1 +
    \omega_n}{n}\right) \le \frac{2\omega_n}{n} \quad \forall j = 1 ,
\ldots , n \, , 
$$
and consequently, by the absolute continuity of $g$,
\begin{equation}\label{eq306}
\frac{2\omega_n}{1 - 2\omega_n} \ge \frac{  g\bigl( ( j- \omega_n)/n
  \bigr) - g \bigl( ( j-1 +
    \omega_n)/ n \bigr)  }{ (1 - 2 \omega_n)/n} = \frac1{\lambda
(J_{n,j})} \int_{J_{n,j}} \!\! g' \, {\rm d}\lambda \quad \forall j =1,
\ldots , n \, .
\end{equation}
Equality holds on the left in (\ref{eq306}) for at least one
$j$, and for that $j$,
$$
\frac{2\omega_n}{1 - 2\omega_n} = \frac1{\lambda
(J_{n,j})} \int_{J_{n,j}} \! \! g' \, {\rm d}\lambda \le \|g'\|_{\infty} \, ,
$$
from which it is clear that
\begin{equation}\label{eq3065}
\omega^+ \le \frac{\|g'\|_{\infty}}{2 + 2 \|g'\|_{\infty}} =
\|\Omega (g')\|_{\infty} \le \frac12 \, .
\end{equation}
Since (\ref{eq305}) trivially holds when $\omega^- = \frac12$,
henceforth assume $\omega^- < \frac12$, and pick $n_1 < n_2 <
\ldots$ so that $\lim_{k \to \infty} \omega_{n_k} =
\omega^-$. Given any $0<x<1$, let $j_k (x) = \lfloor n_k x \rfloor + 1
\in \{1, \ldots, n_k\}$, and note that $J_{n_k, j_k(x)} \subset [x -1/n_k
, x+1/n_k]$, but also
$$
\frac{\lambda (J_{n_k, j_k(x)})}{\lambda ([x - 1/n_k , x+ 1/n_k])} =
\frac{1}{2} - \omega_{n_k} \: \stackrel{k\to \infty}{\longrightarrow} \: \frac12 - \omega^- > 0 \, .
$$
Thus the sequence $(J_{n_k, j_k(x)})_{k\in \N}$ shrinks to $x$ nicely,
and by \cite[Thm.7.10]{Rudin},
$$
g'(x) = \lim\nolimits_{k\to \infty} \frac1{\lambda (J_{n_k , j_k
    (x)})} \int_{J_{n_k, j_k(x)}} \!\!\!\!\!\!\!\! g' \, {\rm d}\lambda \le
\lim\nolimits_{k \to \infty} \frac{2\omega_{n_k}}{1 - 2 \omega_{n_k}}
= \frac{2\omega^-}{1 - 2\omega^-} 
$$
for $\lambda$-almost every $x\in [0,1]$. It follows that $\omega^- \ge
\Omega ( g' )$ holds $\lambda$-almost everywhere, and hence $\omega^- \ge \|\Omega
(g')\|_{\infty}$. Together with (\ref{eq3065}), this completes the proof.
\end{proof}

If $\mu^{-1}$ is singular then the asymptotic behaviour of $\bigl( d_{\epsilon } (\mu,
\delta_{\bullet}^{u_n})\bigr)$ is controlled by a different
mechanism. To prepare for the general result, observe that
$\liminf_{n\to \infty} \mbox{\rm dist} (nx,\Z)=0$ for every $x\in \R$, whereas
\begin{equation}\label{eq307}
\limsup\nolimits_{n\to \infty} \mbox{\rm dist} (nx,\Z) = 
\left\{
\begin{array}{ll}
\frac12 (q-1)/ q & \mbox{\rm if $x =p/q$ with coprime $p\in \Z ,q\in \N$, $q$
  odd,}\\[1mm]
\frac12 & \mbox{\rm otherwise;}
\end{array}
\right.
\end{equation}
in particular, $\limsup\nolimits_{n\to \infty}  \mbox{\rm dist}
(nx,\Z)\ge \frac13 $ unless $x\in \Z$. Defining $\iota : \R \to \N_0
\cup \{ + \infty\} $ as
$$
\iota (x) = 2 \inf \bigl\{ n\in \N_0 : (2n+1) x \in  \Z \bigr\} \, ,
$$
notice that the right-hand side in (\ref{eq307}) is nothing other than
$\Omega \circ \iota (x)$. With this, consider the very simple
example of $\mu_a = a \delta_{-1} + (1-a)\delta_{1}$ for some
$0<a<1$, for which $G_{\mu_a^{-1}} = \{0,a,1\}$. It is readily
confirmed that $nd_{\epsilon } (\mu_a, \delta_{\bullet}^{u_n}) = \mbox{\rm dist} (na,\Z)$
for all sufficiently large $n$, and hence
$\liminf_{n\to \infty} n d_{\epsilon } (\mu_a, \delta_{\bullet}^{u_n}) = 0$,
as well as
$$
\limsup\nolimits_{n\to \infty} n d_{\epsilon } (\mu_a,
\delta_{\bullet}^{u_n}) = \Omega \circ  \iota (a)  = \max
\Omega \circ  \iota (G_{\mu_a^{-1}})  \, .
$$
This equality is but one manifestation of a general principle.

\begin{lem}\label{lem308}
Let $\mu \in \cP$ and $\epsilon >0$. If $\mu^{-1}$ is singular
(w.r.t.\ $\lambda$) then
\begin{equation}\label{eq3081}
\limsup\nolimits_{n\to \infty} n d_{\epsilon } (\mu ,
\delta_{\bullet}^{u_n}) = \sup \Omega \circ \iota (G_{\mu^{-1}}) \, .
\end{equation}
\end{lem}

\begin{proof}
Using the same symbols as in previous proofs, write $G_{\mu^{-1}}$
simply as $G$, and let $2m = \sup \iota (G )$ with the
appropriate $m\in \N_0 \cup \{+\infty\}$; also, let $G^*$
be the set of atoms of $\mu^{-1}$, i.e., $G^* = \bigl\{ 0<x<1 :
g_-(x) < g(x) \bigr\}$. Assume first that $m\in \N_0$. Since $m=0$ implies
$G = \{0,1 \}$, or equivalently $\mu^{-1} = 0$, and (\ref{eq3081})
is correct in this case by Proposition \ref{prop303}, henceforth
assume $m\ge 1$. Then $\mu^{-1}$ is concentrated on
finitely many atoms, thus
$$
G^* = G = \left\{ 
0 , \frac{k_1}{2m_1 + 1} , \ldots , \frac{k_l}{2m_l + 1} , 1
\right\} \, ,
$$
with the appropriate positive integers $l, k_1, \ldots, k_l , m_1,
\ldots, m_l$, where $k_i, 2m_i + 1$ are coprime for all $i$, and
$\max_{i=1}^l m_i = m$. As seen in the example above, for all sufficiently
large $n$,
$$
\omega_n = \max\nolimits_{i=1}^l \mbox{\rm dist} \left(
  \frac{nk_i}{2m_i + 1}, \Z\right) \, ,
$$
and hence
$$
\omega^+ = \max\nolimits_{i=1}^l \Omega \circ \iota \left(
    \frac{k_i}{2m_i + 1}\right) = \max \nolimits_{i=1}^l \Omega (2m_i)
  = \Omega (2m) \, ,
$$
so again (\ref{eq3081}) is correct. It remains to consider the case of
$m=+\infty$. Here it is convenient to consider two subcases, depending
on whether $\iota (G^*)$ is unbounded or not. In the former
case, fix $a\in \R^+$, and pick $x\in G^* $ with $\iota (x)
\ge a$. Moreover, pick $b> 3$, and recall that $y_n := (\mbox{\rm
  dist} (nx, \Z) - 1/b)/n >0$ for infinitely many $n\in \N$. Since $x$
is an atom of $\mu^{-1}$, for every $c\in \R^+$ clearly
\begin{equation}\label{eq3082}
g_- \left( x + \frac1{b n}\right) - g\left( x - \frac1{b n}\right) \ge
\frac{c}{b n} \quad \mbox{\rm for all sufficiently large $n$.}
\end{equation}
Choosing $c= b/\epsilon $ in (\ref{eq3082}), note that for infinitely many
$n$,
\begin{align*}
g_- & \left( \frac{\lfloor nx \rfloor + 1 - n y_n}{n}\right)  - g \left(
  \frac{\lfloor nx \rfloor  + n y_n}{n}\right) = \\
& = g_- \left( 
x + \frac{\max \{1 - 2 \mbox{\rm dist}(nx,\Z), 0\}}{n} + \frac1{b n}
\right) -
g \left( 
x + \frac{\min\{1 - 2 \mbox{\rm dist}(nx,\Z), 0\}}{n} - \frac1{b n}
\right) \\
& \ge g_- \left( x + \frac{1}{b n} \right) - g \left( x -
  \frac1{b n}\right) \ge \frac1{n\epsilon } \ge \frac{2y_n}{\epsilon } \, ,
\end{align*}
and consequently $\omega_n \ge n y_n$. It follows that
$$
\omega^+ \ge \limsup\nolimits_{n\to \infty} \left( \mbox{\rm dist}
  (nx, \Z) - \frac1{b}\right) = \Omega \circ \iota (x)  -
\frac1{b} \ge \Omega (a) - \frac1{b} \, .
$$
Since $a, b> 3$ have been arbitrary, $\omega^+ = \frac12 = \Omega
(2m)$. Finally, assume that $\iota (G^*)$ is bounded, and
hence $G^*$ is finite, possibly empty. Since $m= +\infty$,
clearly $G\setminus G^* \ne \varnothing$, and
every $x \in G\setminus G^* $ is a continuity
point of $g$, as well as an accumulation point of $G$. By
\cite[Thm.7.15]{Rudin},
$$
\lim\nolimits_{\varepsilon \downarrow 0} \frac{g_- (x + \varepsilon) -
  g(x- \varepsilon)}{2\varepsilon} = +\infty \quad \mbox{\rm for
  $\mu^{-1}$-almost every $0<x<1$} \, .
$$
From this it is clear that, given any $b, c>3$, there
exists $x \in G\setminus \Q$ for which (\ref{eq3082})
holds. With $\iota(x) = a = +\infty$, the same argument as before
shows that $\omega^+ = \frac12$, i.e., (\ref{eq3081}) is correct in this
case also.
\end{proof}

Combining Theorem \ref{thm304} and Lemma \ref{lem308} yields a sharp
(upper) rate for $(d_{\epsilon } \bigl(\mu, \delta_{\bullet}^{u_n})\bigr)$,
for arbitrary $\mu \in \cP$. To formulate the result, recall
that every $\sigma$-finite Borel measure $\rho$ on the real line can
be written uniquely as $\rho = \rho_{\sf A} + \rho_{\sf S}$, where
$\rho_{\sf A}$ and $\rho_{\sf S}$ are absolutely
continuous and singular (w.r.t.\ $\lambda$), respectively.

\begin{theorem}\label{thm309}
Let $\mu\in \cP$ and $\epsilon >0$. Then
\begin{equation}\label{eq3091}
\limsup\nolimits_{n\to \infty} n d_{\epsilon } (\mu, \delta_{\bullet}^{u_n})
= \max \left\{
\left\| \, \, \Omega \! \left( \epsilon \frac{{\rm d}(\mu^{-1})_{\sf A}}{{\rm
        d}\lambda}\right) \right\|_{\infty} \!\! , \,  \sup \Omega
\circ \iota (G_{(\mu^{-1})_{\sf S}}) 
\right\} \, .
\end{equation}
\end{theorem}

\begin{proof}
Since there is nothing to prove otherwise, assume that
$(\mu^{-1})_{\sf A}\ne 0$ and $(\mu^{-1})_{\sf S}\ne 0$. In analogy to
the proof of Lemma \ref{lem308}, let $g= F_{\mu}^{-1} = g_{\sf A} +
g_{\sf S}$, with $g_{\sf A}, g_{\sf S} \in \cF$ such that
$\lambda_{g_{\sf A}} = (\mu^{-1})_{\sf A}$ and $\lambda_{g_{\sf S}} =
(\mu^{-1})_{\sf S}$, as well as $2m = \sup \iota (G_{\sf S} )\in \N_0
\cup \{+\infty\}$, where $G_{\sf S} = G_{(\mu^{-1})_{\sf S}}$ for convenience. Since $g_- (y) - g(x) \ge g_{{\sf
    S}-} (y) - g_{\sf S} (x)$ for all $0<x < y < 1$, Lemma
\ref{lem308} yields $\omega^+ \ge \sup \Omega \circ \iota
(G_{\sf S}) = \Omega (2m)$. Thus (\ref{eq3091}) clearly is
correct when $m = +\infty$, and only the case of $m\in \N$ remains
to be considered. (Note that $m=0$ is impossible, as it would imply
$(\mu^{-1})_{\sf S} =0$.) In this case, $G_{\sf S}$ is finite,
say, $G_{\sf S} = \{0,x_1, \ldots, x_l, 1\}$ with $l\in
\N$ and $0<x_1 < \ldots < x_l < 1$. With $J_{n,j}$ as in the proof of Theorem
\ref{thm304}, and letting $K_n = \{ \lfloor n x_i \rfloor : i = 1, \ldots ,
l \}\subset \{ 1, \ldots , n -1 \}$ for $ n\ge 1/x_1$, observe that
$$
g_{\sf A} \left( \frac{j - \omega_n}{n}\right) - g_{\sf A} \left(
  \frac{j-1 + \omega_n}{n}\right) = 
g_- \left( \frac{j - \omega_n}{n}\right) - g  \left(
  \frac{j-1 + \omega_n}{n}\right) \le \frac{2\omega_n }{n\epsilon } \quad
\forall j \not \in K_n \, ,
$$
and consequently
$$
\frac1{\lambda (J_{n,j})} \int_{J_{n,j}} \!\! \epsilon  g_{\sf A}' \, {\rm d}\lambda
\le \frac{2\omega_n }{ 1 - 2 \omega_n} \quad \forall j \not \in
K_n \, .
$$
If $\omega^- < \frac12$ then the same argument as in the proof of
Theorem \ref{thm304} shows that
$$
\epsilon  g_{\sf A}'(x) \le \frac{2\omega^-}{ 1 - 2\omega^-} \quad
\mbox{\rm for $\lambda$-almost every $x$} \, ,
$$
since clearly $j_n(x)\not \in K_n$ whenever $x\not \in G_{\sf S}$ and
$n$ is sufficiently large. Thus $\omega^- \ge \|\Omega  (\epsilon g_{\sf A}')\|_{\infty}$;
trivially, the latter also holds when $\omega^- = \frac12$. In
summary, $\omega^+ \ge \max \bigl\{ \| \Omega  (\epsilon g_{\sf
  A}')\|_{\infty} , \Omega(2m) \bigr\}=: \omega$; note that $\omega$
simply equals the right-hand side in (\ref{eq3091}).

The reverse inequality is non-trivial only when $\omega < \frac12$. In
this case, assume $m\in \N$ as before, and pick
any $z$ with $\omega < z < \frac12$. Then, for all sufficiently large
$n$,
$$
g_{\sf A} \left( \frac{j-z}{n} \right) - g_{\sf A}\left( \frac{j-1 +
    z}{n} \right) \le \frac{2z}{\epsilon  n} \quad \forall j = 1 , \ldots ,
n \, ,
$$
but also, since $G_{\sf S}$ is finite,
$$
g_{{\sf S}-} \left( \frac{j-z}{n} \right) - g_{\sf S}\left( \frac{j-1 +
   z}{n} \right) = 0 \quad \forall j = 1 , \ldots ,
n \, .
$$
Thus $\omega_n \le z$ for all sufficiently large $n$, and since $z>\omega$
was arbitrary, $\omega^+ \le \omega $. 
\end{proof}

\begin{cor}\label{cor310}
Let $\mu\in \cP$ and $\epsilon  >0$.
\begin{enumerate}
\item If $\limsup_{n\to \infty} nd_{\epsilon } (\mu,
  \delta_{\bullet}^{u_n}) < \frac13$ then $\bigl( nd_{\epsilon } (\mu,
  \delta_{\bullet}^{u_n})\bigr)$ converges, and $\mu^{-1}$ is
  absolutely continuous (w.r.t.\ $\lambda$).
\item If $\limsup_{n\to \infty} nd_{\epsilon } (\mu,
  \delta_{\bullet}^{u_n}) \ge \frac13$ then either $\bigl( nd_{\epsilon } (\mu,
  \delta_{\bullet}^{u_n})\bigr)$ converges, or 
$$
\limsup\nolimits_{n\to \infty} n d_{\epsilon } (\mu,
\delta_{\bullet}^{u_n}) \: \in  \: 
\left\{ \frac13 , \frac25 , \frac37 , \ldots , \frac12 \right\} =
\Bigl\{ \Omega  (2m) : m\in \N \cup \{+\infty\} \Bigr\} \, .
$$
\end{enumerate}
\end{cor}

\begin{rem}\label{rem3105}
(i) The proof given above shows that, for every $\mu \in \cP$ and
$\epsilon > 0$,
$$
\liminf\nolimits_{n\to \infty} n d_{\epsilon } (\mu, \delta_{\bullet}^{u_n})
\ge 
\left\| \, \, \Omega \! \left( \epsilon \frac{{\rm d}(\mu^{-1})_{\sf A}}{{\rm
        d}\lambda}\right) \right\|_{\infty}  .
$$
(ii) Let $\mu\in \cP$ be non-atomic. Then the right-hand side in
(\ref{eq3091}) tends to $\frac12$ as $\epsilon \to +\infty$. This is
consistent with the fact that $n \min_{x\in \Xi_n} \|F_{\mu} -
F_{\delta_{x_n}^{u_n}}\|_{\infty} = \frac12$ for all $n\in \N$
whenever $\mu$ is non-atomic \cite[Cor.5.5]{BX}.
\end{rem}

The following example illustrates the results of this section. In
particular, it demonstrates that all situations allowed by Theorem
\ref{thm309} and Corollary \ref{cor310} do occur. It also shows that
(\ref{eq305}) may fail when $\mu^{-1}$
is not absolutely continuous; similarly, (\ref{eq3081}) may fail when
$\mu^{-1}$ is not singular.

\begin{example}\label{ex311}
For every $0\le a<1< b$ consider $\mu_{a,b} = a\delta_{-1} + (1-a)
U_{1,b}$, where $U_{1,b}$ denotes the uniform distribution (normalized
Lebesgue measure) on $[1,b]$. Note that $\mu_{a,1}:= \lim_{b\downarrow
1} \mu_{a,b} = \mu_a$, with $\mu_a$ considered prior to Lemma
\ref{lem308}. Since $(\mu^{-1}_{a,b})_{\sf A} = (b-1) U_{a,1}$ and
$(\mu^{-1}_{a,b})_{\sf S} = 2 \delta_a$ provided that $a>0$, Theorem \ref{thm309}
yields
$$
\limsup\nolimits_{n\to \infty} n d_{\epsilon } (\mu_{a,b},
\delta_{\bullet}^{u_n}) = \max \left\{ \Omega  \left(\epsilon 
    \frac{b-1}{1-a}\right) , \Omega \circ \iota (a) \right\}
\, ,
$$
whereas by direct inspection,
$$
\liminf\nolimits_{n\to \infty} n d_{\epsilon } (\mu_{a,b},
\delta_{\bullet}^{u_n}) = \Omega  \left(\epsilon
    \frac{b-1}{1-a}\right)  \: <  \: \frac12 \, .  
$$
On the one hand, if $a = a_m = m/(2m+1)$ for some $m\in \N_0$ then
$\Omega \circ \iota (a_m) = a_m$, and since $b \mapsto \Omega 
\bigl(\epsilon (b-1)/(1-a_m)\bigr)$ is increasing continuously from $0$ to
$\frac12$, there exists a unique $b_m$ with $ \Omega 
\bigl( \epsilon (b_m-1)/(1-a_m)\bigr)= a_m$. Thus $\bigl( n d_{\epsilon }
(\mu_{a_m,b} , \delta_{\bullet}^{u_n})\bigr)$ converges precisely if
$b\ge b_m$, whereas for $b< b_m$ the lim inf can have any value
between $0$ and $a_m$. On the other hand, if $a=\frac12$ then $\Omega \circ \iota
(a) =\frac12$, and again $\liminf_{n\to \infty} n d_{\epsilon }
(\mu_{1/2,b} , \delta_{\bullet}^{u_n}) = \Omega  \bigl( 2\epsilon  ( b-1)\bigr)$ can
have any value between $0$ and $\frac12$. Except for the case of $\lim_{n\to
  \infty} n d_{\epsilon } (\mu, \delta_{\bullet}^{u_n}) = \frac12$, which
occurs, e.g., for the exponential distributions in Example \ref{ex212}, 
every possible situation allowed by Theorem \ref{thm309} can be
observed in this example by choosing $a,b$ appropriately; see also
Figure \ref{fig1}. Notice that $\mu^{-1}_{a,b}$ is absolutely
continuous precisely if $a=0$, and is singular only if
$b=1$. While Theorem \ref{thm304} and Lemma
\ref{lem308} thus cannot in general be reversed, clearly (\ref{eq305})
and (\ref{eq3081}) may fail if $\mu^{-1}$ is not absolutely continuous and
singular, respectively.
\end{example}

\begin{figure}[ht]
\psfrag{tli}[]{\small $\liminf_{n\to\infty} nd_{\epsilon } (\mu ,
  \delta_{\bullet}^{u_n})$}
\psfrag{tls}[]{\small $\limsup_{n\to\infty} nd_{\epsilon } (\mu ,
  \delta_{\bullet}^{u_n})$}
\psfrag{t0}[]{\small $0$}
\psfrag{t05}[]{\small $\frac12$}
\psfrag{t13}[]{\small $\frac13$}
\psfrag{t25}[]{\small $\frac25$}
\psfrag{t37}[]{\small $\frac37$}
\psfrag{t12}[]{\small $\frac12$}
\psfrag{tdot}[]{\small $\cdots$}
\psfrag{tinf0}[l]{\small $\mu^{-1} = 0$}
\psfrag{tinfd}[]{\small $(\mu^{-1})_{\sf S}=0$}
%
%
\begin{center}
\includegraphics{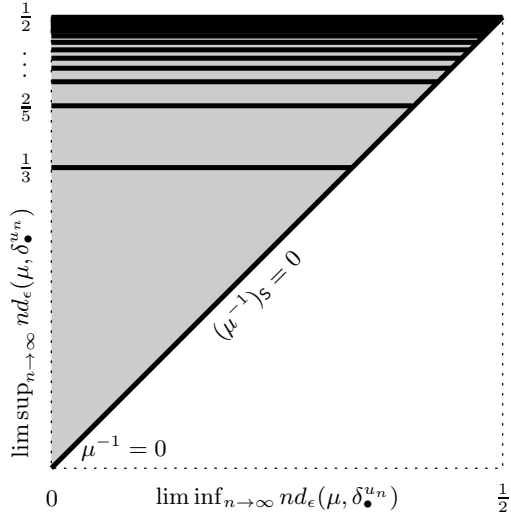}
\end{center}
\caption{Solid black lines indicate, for any $\mu \in \cP$, the possible values
of the limit inferior and the limit superior of $\bigl( n d_{\epsilon } (\mu,
\delta_{\bullet}^{u_n})\bigr)$; see Theorems \ref{thm304}
and \ref{thm309}. Example \ref{ex311} demonstrates that all possible
values may indeed occur.}\label{fig1}
\end{figure}

\begin{example}\label{ex3112}
Let $\mu$ be a normal distribution with
variance $\sigma^2>0$. With $\phi$ denoting the standard normal
distribution function, it is readily deduced from Proposition
\ref{prop209} that $\ell_{\bullet}^{u_n} = d_{\epsilon} (\mu,
\delta_{\bullet}^{u_n})$ is the unique solution of $\phi^{-1} (1-\ell)
- \phi^{-1} (1+\ell - 1/n) = 2\ell / \sqrt{\epsilon^2 \sigma^2}$. Utilizing the
familiar fact \cite[Sec.26.2]{AS}
$$
1 - \phi(x) = \frac{e^{-x^2/2}}{\sqrt{2\pi }} \left( x^{-1} - x^{-3} +
\cO (x^{-5})\right) \quad \mbox{\rm as } x \to +\infty \, ,
$$
a straightforward analysis of this equation yields
\begin{equation}\label{eq3115}
n d_{\epsilon} (\mu, \delta_{\bullet}^{u_n}) = \frac12 -
\frac{1}{2\epsilon \sqrt{2\sigma^2} } \cdot \frac{\sqrt{\log n}}{n} +
\cO (n^{-1}) \quad \mbox{\rm as } n \to \infty \, ,
\end{equation}
which sharpens (\ref{eq305}). Notice that convergence occurs here at a slightly slower rate than has so far
been observed in this article for examples of absolutely continuous $\mu^{-1}$; cf.\ \cite{GHMR} and \cite[Ex.5.18]{XB}.
\end{example}

As Examples \ref{ex212}, \ref{ex216}, and \ref{ex3112} suggest,
the results of this section, notably Theorem \ref{thm304}, can be refined by imposing
further assumptions on $\mu$. For instance, assume that $g =
F_{\mu}^{-1}$ is $C^2$ on $]0,1[$, and that $g,g'\ne 0$ both are convex, with
$\limsup_{x\uparrow 1} (1-x) g''(x)/g'(x)<+\infty$. (All mentioned
examples meet these requirements; for normal distributions, only the
right half of $]0,1[$ has to be considered due to symmetry.) Letting $c =
\|\Omega  (\epsilon g')\|_{\infty}>0$ for convenience, it is straightforward
to show that, as a refinement of (\ref{eq305}),
\begin{equation}\label{eq313}
n d_{\epsilon } (\mu, \delta_{\bullet}^{u_n}) = c - \frac{2c^2}{\epsilon }
e_n   + o(e_n)  \quad \mbox{\rm as $n\to \infty$,}
\end{equation}
where $\lim_{n\to \infty} e_n = 0$, and more specifically,
$$
e_n = \left\{ \begin{array}{ll}
\displaystyle \frac1{g_-'(1)} + n  \frac{g(  1 - (1-c)/n) - g(1-c/n)}{(1-2c)
  g_-'(1)^2}  & \mbox{\rm if } c < \frac12 \, , \\[4mm]
\displaystyle \frac1{g' \bigl(1 - 1/(2n)\bigr)} & \mbox{\rm if } c =
\frac12 \, ;
\end{array}
\right.
$$
in particular, if $g_-''(1)< + \infty$ then simply
$$
e_n = \frac{g_-''(1)}{2 g_-'(1)^2} \, n^{-1} \, .
$$
As the reader may want to check, for exponential, Benford, and normal
distributions, the asymptotic equalities (\ref{eq214}), (\ref{eq217}), and (\ref{eq3115}),
respectively, all are (slightly sharper than, but certainly)
consistent with (\ref{eq313}).

\begin{example}\label{ex320}
Let $\mu$ be the {\em Cantor distribution}, i.e., the $\log 2/ \log
3$-dimensional Hausdorff measure on the classical Cantor middle-thirds
set $C$. Thus $G_{\mu} = C$, and since $\mbox{\rm diam}\, C = 1$, the
measure $\mu^{-1}$, referred to as the {\em inverse Cantor
  distribution\/} \cite[Ex.A.11]{BL}, is a probability measure as well. Both
$\mu, \mu^{-1}$ are singular: While $\mu$ is non-atomic, $\mu^{-1}$ is purely atomic; in fact, $\mu^{-1}(\{ i
2^{-m}\}) = 3^{-m}$ for every $m\in \N$ and every odd $1\le i <
2^m$. Obviously, Lemma \ref{lem308} applies to both distributions,
showing that
$$
\limsup\nolimits_{n\to \infty} n d_{\epsilon} ( \mu,
\delta^{u_n}_{\bullet}) =
\limsup\nolimits_{n\to \infty} n d_{\epsilon} ( \mu^{-1},
\delta^{u_n}_{\bullet}) =  \frac12 \, .
$$
By contrast, it is straightforward to check that $\liminf_{n\to
  \infty} n d_{\epsilon} ( \mu,
\delta^{u_n}_{\bullet}) = 0$. For the inverse Cantor distribution, an
elementary analysis \cite{BX, XThesis} yields $\frac1{216} \le \liminf_{n\to
  \infty} n d_1 ( \mu^{-1}, \delta^{u_n}_{\bullet}) \le \frac13$, but the authors do not know the
precise value of $\liminf_{n\to
  \infty} n d_{\epsilon} ( \mu^{-1},
\delta^{u_n}_{\bullet})$ for any $\epsilon > 0$.
\end{example}

\section{Best (unconstrained) L\'evy approximations}\label{sec4}

This section studies the asymptotics of $d_{\epsilon} (\mu,
\delta_{\bullet}^{\bullet, n})$ as $n\to \infty$. The following theorem is the
section's main result and a counterpart to Theorems \ref{thm304} and
\ref{thm309}. It asserts that the sequence $\bigl( n d_{\epsilon} (\mu,
\delta_{\bullet}^{\bullet, n})\bigr)$ always converges, to a limit
smaller than $\frac12$ that is easily expressed in terms of $\Omega$
and $(\mu^{-1})_{\sf A}$, the absolutely continuous part of
$\mu^{-1}$.

\begin{theorem}\label{thm450}
Let $\mu \in \cP$ and $\epsilon > 0$. Then
\begin{equation}\label{eq4501}
\lim\nolimits_{n\to \infty} n d_{\epsilon} (\mu,
\delta_{\bullet}^{\bullet, n}) = \int \Omega \! \left( \epsilon 
  \frac{{\rm d} (\mu^{-1})_{\sf A}}{{\rm d} \lambda}\right) {\rm d}
\lambda  \, .
\end{equation}
\end{theorem}

\noindent
 As pointed out already in
the Introduction, Theorem \ref{thm450} may be regarded as an analogue for
$d_{\epsilon}$ of (the one-dimensional version of) a classical $d_{\sf W}$-quantization theorem
\cite[Thm.6.2]{GL}, but unlike that theorem, it does not impose a
moment assumption on $(\mu^{-1})_{\sf A}$. For $d_{\sf
  W}$-approximations, such an assumption is known to be essential, as
the $d_{\sf W}$-approximation error for $\mu \in \cP_1$ can decay
arbitrarily slowly without it; see \cite[Ex.6.4]{GL} and
\cite[Thm.5.33]{XB}. In the light of Theorem \ref{thm450}, this
may be viewed as an artefact specific to $d_{\sf W}$ that
does not exist for $d_{\epsilon}$. Also, recall that unlike
$d_{\sf W}$, the metric $d_{\epsilon}$ metrizes precisely the topology
of weak convergence on all of $\cP$, and so does $d_{\sf P}$. As far
as the authors have been able to ascertain, however, all known
results pertaining to the asymptotics of the $d_{\sf P}$-approximation
error also impose additional assumptions \cite[Sec.4]{GL1}, and
despite the similarities between $d_{\sf P}$ and $d_{\epsilon}$,
\cite[Sec.5]{GL1} suggests that the $d_{\sf P}$-approximation error
can decay arbitrarily slowly as well.

When proving Theorem \ref{thm450}, the following observation, a direct consequence of Proposition \ref{prop209}
together with the argument establishing \cite[Thm.3.9]{BX}, is
helpful; its routine verification is left to the interested reader.

\begin{prop}\label{prop4502}
Let $\mu\in \cP$, $\epsilon > 0$, and $n\in \N$. With $\ell =
d_{\epsilon} (\mu, \delta_{\bullet}^{\bullet, n})$, there exists $p\in
\Pi_n$ such that for every $j=1, \ldots , n$,
$$
\mbox{\rm dist}(P_{,j}, G_{\mu^{-1}}) \le \ell \quad \mbox{\it and}
\quad
\mu^{-1} \bigl( ]P_{,j-1} +\ell, P_{,j} - \ell[ \bigr) \le \frac{2\ell}{\epsilon}
\le
\mu^{-1} \bigl( [P_{,j-1} +\ell, P_{,j} - \ell] \bigr) \, .
$$
\end{prop}

\begin{proof}[Proof of Theorem \ref{thm450}]
Throughout this proof, for convenience let $g=F_{\mu}^{-1}$ and $G = G_{\mu^{-1}}$ as before,
but also $\ell_n = d_{\epsilon} (\mu, \delta_{\bullet}^{\bullet, n})$
and $\omega_n = n \ell_n$ for all $n\in \N$, as well as $\omega^- = \liminf_{n\to \infty} \omega_n$ and $\omega^+ =
\limsup_{n\to \infty} \omega_n$. Again it suffices to consider the case of
$\epsilon = 1$. Note that $\ell_n = 0$ for some (and hence all
sufficiently large) $n\in \N$ if and only if $G$ is finite,
in which case (\ref{eq4501}) clearly is correct. Thus assume
$G$ to be infinite from now on, and consequently $\ell_n > 0$ for
all $n\in \N$.

Given $n\in \N$, choose $p_n \in \Pi_n$ as in Proposition
\ref{prop4502}, and notice that $\ell_n > 0$ implies $\min_{j=1}^n
(P_{n,j} - P_{n,j-1})\ge 2 \ell_n > 0$; in particular, $P_{n,j-1} <
P_{n,j}$ for all $j=1,\ldots , n$. Consequently, for every $x\in
[0,1[$ there exists a unique $j_n(x)\in \{1, \ldots , n\}$ with $P_{n,
j_n(x)-1} \le x < P_{n, j_n (x)}$. For convenience, let $J_{n,j} =
[P_{n,j-1} + \ell_n , P_{n,j} - \ell_n]$ for all $j=1, \ldots ,
n$, and hence $\lambda_{n,j} := \lambda (J_{n,j}) = P_{n,j} -
P_{n,j-1} - 2\ell_n$. Next, recall that the set $U:= [0,1]\setminus G$ is
open, possibly empty. If $U \ne \varnothing$ let $I_1, I_2,
\ldots$ be its (at most countably many) connected components, that is,
the disjoint open intervals with endpoints in $G$ and $U = \bigcup_{k}
I_k$. Thus, for every $x\in U$ there exists a unique $k(x)\in \{1,2,
\ldots \}$
with $x\in I_{k(x)}$. Finally, consider the subset
$G^{\dagger}$ of $G$ defined as
\begin{equation}\label{eq4503Z}
G^{\dagger} = \bigl\{x\in G : I \cap G  =
  \{x\}  \: \mbox{\rm for some interval $I$ with $\lambda(I)>0$} \bigr\} \, .
\end{equation}
Notice that $\{0,1\}\subset G^{\dagger}$, and
$G^{\dagger}$ is (at most) countable. Utilizing Proposition
\ref{prop4502}, it is readily checked that
\begin{equation}\label{eq4503A}
\lim\nolimits_{n\to \infty} [P_{n, j_n(x) -1}, P_{n,j_n(x)}] = \{x\}
\quad \forall x \in G\setminus G^{\dagger} \, ,
\end{equation}
but also
\begin{equation}\label{eq4503B}
\lim\nolimits_{n\to \infty} [P_{n, j_n(x) -1}, P_{n,j_n(x)}] =
I_{k(x)} \quad \forall x \in U \, .
\end{equation}
With these preparations, the proof is now carried out in three separate
steps for the reader's convenience.

\medskip
\noindent
{\em Step I: Assume $\mu^{-1}$ is absolutely continuous.}
\smallskip

\noindent
Proposition \ref{prop4502} with $\epsilon=1$ yields $\mu^{-1}
(J_{n,j}) = 2\ell_n$ or, equivalently,
\begin{equation}\label{eq4504}
g(P_{n,j} - \ell_n) - \ell_n = g (P_{n,j-1} + \ell_n) + \ell_n \quad
\forall j = 1 , \ldots , n \, ,
\end{equation}
and since $\mu^{-1}$ is absolutely continuous also $\lambda_{n,j} >
0$. Fix any $0< a < 1$, and recalling that $g$ is differentiable
$\lambda$-almost everywhere, with $g'\ge
0$ integrable over every compact subinterval of $]0,1[$, pick a
non-negative continuous function $f_a: \:  ]0,1[ \: \to \R$ with $\int_{[0,1]}
|g' - f_a|\, {\rm d}\lambda < a$. (Notice that $g',f_a$ may not be
integrable over $[0,1]$.)
If $U \ne \varnothing$ then also pick $k_a \in \N$ large enough
to ensure $\lambda ( \bigcup_{k>k_a} I_k) < a$, and for every $k=1, \ldots
, k_a$ pick a continuous function $e_k : \: ]0,1[ \: \to [0,1]$ with $e_k(x)
= 1$ for all $x\in\:  ]0,1[\setminus I_k$ such that $\int_{I_k} f_a e_k <
a\lambda (I_k)$. Let $f = f_a \prod_{k=1}^{k_a} e_k$; in case
$U=\varnothing$ simply let $f=f_a$. Clearly, $f$
is non-negative and continuous on $]0,1[$, with
\begin{equation}\label{eq4504A}
\int_{[0,1]} |g' - f|\, {\rm d}\lambda = \int_{G} |g' -
f_a|\, {\rm d}\lambda + \int_{U} |f|\, {\rm d}\lambda \le
\int_{[0,1]} |g' - f_a|\, {\rm d}\lambda < a \, ,
\end{equation}
since $g'$ vanishes on $U$. Next, deduce from (\ref{eq4504}) that
$$
\frac1{\lambda_{n,j}} \int_{J_{n,j}} f \, {\rm d}\lambda =
\frac{2\ell_n}{\lambda_{n,j}}  +
\frac1{\lambda_{n,j}} \int_{J_{n,j}} (f - g')  \, {\rm d}\lambda =
\frac{2\ell_n - \displaystyle \int_{J_{n,j} } (g' - f)\, {\rm d} \lambda}{P_{n,j} -
  P_{n,j-1} - 2 \ell_n} \quad \forall j=1 , \ldots , n \, ,
$$
and consequently
\begin{equation}\label{eq4505}
2 \, \Omega \! \left( 
\frac1{\lambda_{n,j}} \int_{J_{n,j}} f \, {\rm d}\lambda \right) 
(P_{n,j} - P_{n,j-1}) = 2\ell_n -  \frac{\displaystyle \int_{J_{n,j}} (g' - f)\, {\rm
    d}\lambda}{1 + \displaystyle  \frac1{\lambda_{n,j}} \int_{J_{n,j}} f \, {\rm
    d}\lambda }
\quad \forall j = 1, \ldots , n \, .
\end{equation}
Summing (\ref{eq4505}) over $j=1, \ldots , n$ yields
\begin{equation}\label{eq4506}
\omega_n -  \int_{[0,1]} h_n \, {\rm d}\lambda = \frac12 \sum\nolimits_{j=1}^n   \frac{\displaystyle \int_{J_{n,j}} (g' - f)\, {\rm
    d}\lambda}{1 + \displaystyle  \frac1{\lambda_{n,j}} \int_{J_{n,j}} f \, {\rm
    d}\lambda } \, ,
\end{equation}
with the piecewise constant non-negative function $h_n:[0,1[\: \to \R$ given by
$$
h_n(x) = \Omega \left( 
\frac1{\lambda_{n,j_n(x)}} \int_{J_{n,j_n(x)}} f \, {\rm d}\lambda
\right) \quad \forall x \in [0,1[ \, .
$$
Recall that $f\ge 0$, and so the right-hand side in (\ref{eq4506})
is bounded, for every $n\in \N$, by $\frac12 \sum_{j=1}^n \int_{J_{n,j}} |g'- f|\, {\rm
  d}\lambda \le \frac12 \int_{[0,1]} |g' - f|\, {\rm d}\lambda <
\frac12 a$. Deduce from (\ref{eq4503A}) and the continuity of $f$ that
\begin{equation}\label{eq4508}
\lim\nolimits_{n\to \infty} h_n (x) = \Omega \bigl( f(x)\bigr) \quad
\forall x \in G \setminus G^{\dagger} \, .
\end{equation}
Similarly, (\ref{eq4503B}) and the choice of the functions $e_k$ for
$k=1, \ldots, k_a$ imply that
\begin{equation}\label{eq4509}
\lim\nolimits_{n\to \infty} h_n(x) = \Omega \left( \frac1{\lambda
    (I_{k(x)})} \int_{I_{k(x)}} f \, {\rm d}\lambda \right) \le \Omega
(a) < \frac{a}{2} \quad \forall x \in
\bigcup\nolimits_{k=1}^{k_a} I_k \, .
\end{equation}
The elementary estimate, valid for all $n\in \N$,
\begin{align*}
\left|
\omega_n - \int_{[0,1]} \Omega (g') \, {\rm d}\lambda
\right| & \le \left|
\omega_n - \int_{[0,1]} h_n  \, {\rm d}\lambda
\right| + \int_{G} |h_n - \Omega (f)|\, {\rm d}\lambda \: +
\\ & \quad + \: \int_{[0,1]} |\Omega (f) - \Omega (g')| \, {\rm d}\lambda +
 \int_{\bigcup_{k=1}^{k_a} I_k} h_n \, {\rm d}\lambda +
\int_{\bigcup_{k> k_a} I_k} h_n \, {\rm d}\lambda \, ,
\end{align*}
together with (\ref{eq4506}), the Dominated Convergence Theorem and
(\ref{eq4508}), the estimate (\ref{eq4504A}), Fatou's lemma and
(\ref{eq4509}), as well as the choice of $k_a$ and the fact that $0\le h_n  \le \frac12$, yield
$$
\limsup\nolimits_{n\to \infty} \left| \omega_n - \int_{[0,1]} \Omega
  (g') \, {\rm d}\lambda \right| \le \frac{a}{2} + 0 +
\frac{a}{2} +\frac{ a}{2} + \frac{a}{2}  = 2a \, .
$$
Since $0<a<1$ has been arbitrary, $\lim_{n\to \infty} \omega_n =
\int_{[0,1]} \Omega (g') \, {\rm d}\lambda$, i.e., (\ref{eq4501}) holds.

\medskip
\noindent
{\em Step II: Assume $\mu^{-1}$ is singular.}
\smallskip

\noindent
Given any $0<a<1$, let $U_a = \{x \in [0,1] : \mbox{\rm dist}\, (x, G)
\ge a\}$. Note that $U_a\subset [a,1-a]$ is a compact, possibly empty subset of $U$,
so $U_a \cap G = \varnothing$. Assume for the time being that all
atoms of $\mu^{-1}$ in $[a,1-a]$ are small in that
\begin{equation}\label{eq4601}
\mu^{-1} (\{x\}) \le a^2 \quad \forall x \in [ a, 1-a ]  \, .
\end{equation}
Recall that $2\ell_n \le \mu^{-1} (J_{n,j})$ for all $j=1, \ldots ,
n$, by Proposition \ref{prop4502}, and correspondingly
$$
\omega_n \le \sum\nolimits_{j=1}^n \Omega \left( \frac{\mu^{-1}
    (J_{n,j})}{\lambda_{n,j}}\right) (P_{n,j} - P_{n,j-1}) =
\int_{[0,1]} \widetilde{h}_n \, {\rm d}\lambda \, ,
$$
with the piecewise constant function $\widetilde{h}_n : [0,1[\: \to
\R$ given by
$$
\widetilde{h}_n (x) = \Omega \left( \frac{\mu^{-1}
    (J_{n,j_n(x)})}{\lambda_{n, j_n(x)}} \right) \quad \forall x \in
[0,1[ \, .
$$
First, observe that if $x\in U_a$ then (\ref{eq4503B}) and
(\ref{eq4601}) imply that $\limsup_{n\to \infty} \mu^{-1} (J_{n,j_n(x)})
\le \mu^{-1} (I_{k(x)}) \le 2 a^2$, whereas clearly $\lim_{n\to \infty}
\lambda_{n,j_n(x)} = \lambda (I_{k(x)})\ge 2 a$. Thus
\begin{equation}\label{eq4602}
\limsup\nolimits_{n\to \infty} \widetilde{h}_n(x) \le \Omega (a) <
\frac{a}{2} \quad \forall x \in U_a \, .
\end{equation}
Next, notice that if $x\in G\setminus G^{\dagger}$ then $([P_{n,j_n(x)
  -1}, P_{n,j_n(x)}])$ shrinks to $x$ nicely, and hence
$$
\lim\nolimits_{n\to \infty} \frac{\mu^{-1} ([P_{n,j_n(x) -1},
  P_{n,j_n(x)}])}{P_{n,j_n(x)} - P_{n,j_n(x) -1}} = 0 \quad \mbox{\rm
  for $\lambda$-almost every } x\in  G \, ,
$$
by \cite[Thm.7.13]{Rudin}. Thus $\lim_{n\to \infty} \ell_n/(
P_{n,j_n(x)} - P_{n,j_n(x)-1}) = 0$ for $\lambda$-almost every $x\in
G$, which in turn shows that $(J_{n,j_n(x)})$ shrinks to $x$ nicely as
well. Applying \cite[Thm.7.13]{Rudin} once more yields $\lim_{n\to
  \infty} \mu^{-1} (J_{n,j_n(x)})/\lambda_{n,j_n(x)} =0$ for
$\lambda$-almost every $x\in G$, and thus
\begin{equation}\label{eq4603}
\lim\nolimits_{n\to \infty} \widetilde{h}_n(x) = 0 \quad
\mbox{\rm for $\lambda$-almost every } x\in G \, .
\end{equation}
Recalling that $G^{\dagger}$ is countable, deduce from (\ref{eq4602})
and (\ref{eq4603}) that
\begin{align}\label{eq4604}
\omega^+ \le \limsup\nolimits_{n\to \infty} \int_{[0,1]}
\widetilde{h}_n \, {\rm d}\lambda & = \limsup\nolimits_{n\to \infty}
\left( \int_{U_a} \widetilde{h}_n \, {\rm d}\lambda +   \int_{G}
\widetilde{h}_n \, {\rm d}\lambda  + \int_{[0,1] \setminus ( U_a \cup G)}
\widetilde{h}_n \, {\rm d}\lambda \right) \nonumber \\
& \le \frac{a}{2} +   \frac12 \bigl( 1 - \lambda (U_a \cup G )\bigr) \, .
\end{align}
In summary, (\ref{eq4604}) holds provided that $\mu$ satisfies
(\ref{eq4601}).

To conclude the argument in the case of $\mu^{-1}$ being singular, given
$0<b<1$, pick $0<a < b$ so small that $\lambda (U_a \cup G ) > 1 -
b$. Noting that the set $G_a:=
\bigl\{ x \in  [a,1-a]  : \mu^{-1} (\{x\}) > a^2 \bigr\}$ is
finite, consider $\widetilde{g}\in \cF$ given by
$$
\widetilde{g} = g - \sum\nolimits_{x \in G_a} \mu^{-1} (\{ x\}) {\bf
  1}_{[x, +\infty[} \, ,
$$ 
as well as the unique $\widetilde{\mu}\in \cP$ with
$F_{\widetilde{\mu}}^{-1} = \widetilde{g}$. Crucially, (\ref{eq4601})
holds with $\widetilde{\mu}$ instead of $\mu$. Moreover, notice that
$\widetilde{G}:= G_{\widetilde{\mu}^{-1}} \supset G \setminus G_a$,
and clearly $\widetilde{U}_a \supset U_a$, where $\widetilde{U}_a =
\{x\in [0,1] : \mbox{\rm dist}\, (x, \widetilde{G}) \ge a\}$. Thus
$\widetilde{U}_a \cup \widetilde{G} \supset ( U_a \cup G) \setminus G_a$, and (\ref{eq4604}) applied to
$\widetilde{\mu}$, with $\widetilde{\ell}_n := d_1 (\widetilde{\mu},
\delta_{\bullet}^{\bullet, n})$, yields
$$
\limsup\nolimits_{n\to \infty} n \widetilde{\ell}_n \le  \frac{a}{2} +
\frac12 \bigl( 1 - \lambda( \widetilde{U}_a \cup\widetilde{ G} )\bigr) < b \, . 
$$
Finally, let $m_a = \# G_a$ and observe that $\ell_{n + m_a} \le \widetilde{\ell}_n$ for
all $n\in \N$, so
$$
\omega_n \le n \widetilde{\ell}_{n - m_a} = (n - m_a)
\widetilde{\ell}_{n - m_a} + m_a \widetilde{\ell}_{n - m_a}
$$
for all $n> m_a$. Since $\lim_{n\to \infty} \widetilde{\ell_n} =0$,
clearly $\omega^+ \le \limsup_{n\to \infty} n \widetilde{\ell}_n <
b$, and since $0<b<1$ has been arbitrary, $\omega^+ = 0$. Thus
(\ref{eq4501}) holds, with vanishing right-hand side, whenever
$\mu^{-1}$ is singular.

\medskip
\noindent
{\em Step III: Let $\mu\in \cP$ be arbitrary.}
\smallskip

\noindent
As in the proof of Theorem \ref{thm309}, write $g = g_{\sf A} + g_{\sf
S}$ with $g_{\sf A}, g_{\sf S} \in \cF$ such that $\lambda_{g_{\sf A}}
= (\mu^{-1})_{\sf A}$ and $\lambda_{g_{\sf S}} = (\mu^{-1})_{\sf
  S}$. Let $\mu^{\langle {\sf A}\rangle}$ and $\mu^{\langle{\sf S}\rangle}$ be the (uniquely
determined) probability measures with $(\mu^{\langle {\sf A}\rangle})^{-1} =
(\mu^{-1})_{\sf A}$ and $(\mu^{\langle {\sf S}\rangle})^{-1} =
(\mu^{-1})_{\sf S}$, respectively. (Notice that in general $\mu^{\langle {\sf A}\rangle}\ne \mu_{\sf A}$ and $\mu^{\langle {\sf S}\rangle}\ne \mu_{\sf S}$.)
Also, for every $n \in \N$ let $\ell_n^{\langle {\sf A} \rangle}= d_1
( \mu^{\langle {\sf A}\rangle} , \delta_{\bullet}^{\bullet, n})$ and
$\ell_n^{\langle {\sf S} \rangle}= d_1 (\mu^{\langle {\sf S}\rangle} ,
\delta_{\bullet}^{\bullet, n})$. Given any $m,n\in \N$,
pick $p_m^{\langle {\sf A}\rangle} \in \Pi_m$
and $p_n^{\langle {\sf S}\rangle}\in \Pi_n$ as in Proposition
\ref{prop4502}. By considering the joint
partition of $[0,1]$ generated by $\left\{ P_{m,i}^{\langle {\sf A}\rangle} : i=0,
\ldots , m\right\}$ and $\left\{P_{n,j}^{\langle {\sf S}\rangle } : j = 0, \ldots , n\right\}$, it is
readily seen that $\ell_{m+n} \le \ell_m^{\langle {\sf
    A}\rangle} + \ell_n^{\langle {\sf S} \rangle}$. For every $0<a<1$ and $n\in \N$, therefore
$$
\omega_n \le n \ell_{\lfloor (1-a)n \rfloor + \lfloor an \rfloor} \le
\frac{1 + \lfloor (1-a) n \rfloor}{1 - a} \ell_{\lfloor
  (1-a) n\rfloor}^{\langle {\sf A}\rangle} + \frac{1 +\lfloor an
  \rfloor}{a} \ell_{  \lfloor an \rfloor}^{\langle {\sf S}\rangle} \, ,
$$
and applying Steps I and II to $\mu^{\langle {\sf A}\rangle }$ and
$\mu^{\langle {\sf
    S}\rangle }$, respectively, yields
$$
\omega^+ \le \frac1{1-a} \int \Omega \left(
 \frac{{\rm d} (\mu^{\langle {\sf A}\rangle})^{-1}_{\sf A}}{{\rm d} \lambda}
\right) 
{\rm d}\lambda = \frac1{1-a} \int \Omega (g_{\sf A}') \, {\rm d}\lambda \, .
$$
(Recall that $\lim_{n\to \infty} \ell_n^{\langle {\sf A} \rangle} = \lim_{n\to
  \infty} n \ell_n^{\langle {\sf S} \rangle} =0$.) Since $0<a<1$ has been arbitrary,
$\omega^+\le \int \Omega (g_{\sf A}') \, {\rm d}\lambda$. To obtain a
lower bound for $\omega^-$, recall from Proposition \ref{prop4502}
that
$$
g_- (P_{n,j} - \ell_n) - \ell_n \le g(P_{n,j-1} + \ell_n) + \ell_n
\quad \forall j = 1, \ldots , n \, ,
$$
and since $g_- = g_{\sf A} + g_{{\sf S}-}$,
\begin{align*}
g_{\sf A} (P_{n,j} - \ell_n) - \ell_n  & \le g_{\sf A} (P_{n, j-1} +
\ell_n) + \ell_n - \max \bigl\{ 0 , \bigl( g_{{\sf S}-} (P_{n,j}- \ell_n ) - g_{\sf S}
(P_{n,j-1}+ \ell_n)\bigr) \bigr\}  \\
& \le g_{\sf A} (P_{n, j-1} +
\ell_n) + \ell_n \quad \forall j = 1, \ldots , n \, ,
\end{align*}
from which it is clear that $\ell_n^{\langle {\sf A} \rangle } \le \ell_n$ for every
$n\in \N$. Applying Step I to $\mu^{\langle {\sf A} \rangle}$ yields $\omega_- \ge
\liminf_{n\to \infty} n \ell_n^{\langle {\sf A} \rangle} = \int \Omega (g_{\sf
  A}')\, {\rm d}\lambda $. Hence $\lim_{n\to \infty} \omega_n = \int \Omega (g_{\sf
  A}')\, {\rm d}\lambda $, and the proof is complete.
\end{proof}

Along the lines of the above proof, and by considering the absolutely
continuous part of $\mu$ rather than of $\mu^{-1}$, the following dual version of
Theorem \ref{thm450} can be established; the routine details are left
to the interested reader.

\begin{prop}\label{prop454}
Let $\mu \in \cP$ and $\epsilon > 0$. Then
$$
\lim\nolimits_{n\to \infty} n d_{\epsilon} (\mu,
\delta_{\bullet}^{\bullet, n}) = \epsilon \int \Omega \! \left(
  \frac1{\epsilon} \cdot  
  \frac{{\rm d} \mu_{\sf A}}{{\rm d} \lambda}\right) {\rm d}
\lambda  \, .
$$
\end{prop}

\noindent
Notice that Theorem \ref{thm450} and Proposition \ref{prop454}
together imply the familiar fact \cite[Ch.A]{BL} that  $\mu^{-1}$ is
singular if and only if $\mu$ is singular, and hence yield a direct analogue of
Proposition \ref{prop303} in the context of best approximations.

\begin{prop}\label{cor455}
For every $\mu \in \cP$ and $\epsilon > 0$, the following are
equivalent:
\begin{enumerate}
\item $\lim_{n\to \infty} n d_{\epsilon} (\mu,
  \delta_{\bullet}^{\bullet, n}) = 0$;
\item $\mu_{\sf A} = 0$;
\item $(\mu^{-1})_{\sf A} = 0$.
\end{enumerate}
\end{prop}

\noindent
Just as in the case of best uniform approximations, Theorem
\ref{thm450} can be refined through further assumptions on $\mu$. For
instance, if $g = F_{\mu}^{-1}$ is $C^4$ on $]0,1[$, if both $g,g'\ne 0$ are
convex, and if $\Omega  (\epsilon g')$ has a $C^3$-extension to $\R$,
then mild boundedness assumptions on $g$ and its derivatives (ensuring
all relevant integrals are finite) guarantee that, as a refinement of (\ref{eq4501}),
\begin{equation}\label{eq4om2}
nd_{\epsilon} (\mu, \delta_{\bullet}^{\bullet, n}) = c_1 + \frac{c_1^2
c_2}{12} \, n^{-2} + o (n^{-2}) \quad \mbox{\rm as } n \to \infty \, ,
\end{equation}
where $c_1 = \int \Omega  (\epsilon g') \, {\rm d}\lambda$ and
$$
c_2 = \int  \frac{2 (1+\epsilon g') (g'')^2- (2 + \epsilon
  g')g' g''' }{ (1 + \epsilon g')^2(g')^2} \, {\rm d}\lambda \, .
$$
When compared to $\bigl( n d_{\epsilon} (\mu,
\delta_{\bullet}^{u_n})\bigr)$, therefore, not only does the sequence $\bigl( n d_{\epsilon} (\mu,
\delta_{\bullet}^{\bullet , n})\bigr)$ converge to a smaller value
(unless $g'$ is constant), but also it converges at the rate $(n^{-2})$ which often is faster than
the rate in (\ref{eq313}). For example, (\ref{eq4om2})
applies to exponential as well as Benford distributions, and the
reader may want to check that (\ref{eq2155}) and (\ref{eq218}) both
are consistent with it. If $\mu$ is a normal distribution with variance $\sigma^2 >
0$ then $\Omega (\epsilon g')$ does not have even a $C^1$-extension to $\R$, and
correspondingly $c_2 = -\infty$, which suggests that $\bigl( n d_{\epsilon} (\mu,
\delta_{\bullet}^{\bullet, n})\bigr)$ converges at a slower
rate. This indeed is the case: An elementary albeit lengthy analysis yields
$$
n d_{\epsilon} (\mu, \delta_{\bullet}^{\bullet , n}) =  - \epsilon
\sqrt{\frac{\pi \sigma^2}{2}} \, {\rm Li}_{1/2} \! \left( -
  \frac1{\epsilon \sqrt{2\pi \sigma^2}} \right)+
\cO \left( \frac{\log n }{n^2}\right) \quad \mbox{\rm as } n \to \infty
\, ,
$$
where ${\rm Li}_{1/2}$ denotes the polylogarithm of order $\frac12$; see, e.g., \cite[\S 25.12]{OLBC}.
Though slower than (\ref{eq4om2}), this rate of convergence again is considerably
faster than its counterpart (\ref{eq3115}) for best uniform approximations.
It should be noted, however, that such a hierarchy of rates, though observed
for many familiar distributions, is by no means universal: As
mentioned already in the Introduction, for the $1$-Pareto distribution
both sequences $\bigl( n d_{\epsilon} (\mu, \delta_{\bullet}^{u_n})\bigr)$
and $\bigl( n d_{\epsilon} (\mu, \delta_{\bullet}^{\bullet ,
  n})\bigr)$ converge to their respective limits $\frac12$ and $\frac{\pi}{8}$
at the same rate $(n^{-2})$, as is evident from (\ref{eq1om1}) and
(\ref{eq1om2}). For the $\frac12$-Pareto distribution,
i.e., $F_{\mu} (x) = 1 - x^{-1/2}$ for all $x\ge 1$,
(\ref{eq313}) yields
$$
n d_1 (\mu , \delta_{\bullet}^{u_n}) = 0.5000 - 0.03125 \, n^{-3} + o
(n^{-3}) \quad \mbox{\rm as } n \to \infty \, ,
$$
whereas (\ref{eq4om2}), with $c_1 =  \int_0^1 (2+t^3)^{-1} {\rm d}t =
0.4508$ and $c_2 = 6 \int_0^1 t(2- t^3)(2+t^3)^{-2} {\rm d}t = 0.9102$, reads
$$
n d_1 (\mu, \delta_{\bullet}^{\bullet , n}) = 0.4508  + 0.01541  \, n^{-2} +
o (n^{-2}) \quad \mbox{\rm as } n \to \infty \, .
$$
Here $\bigl( n d_1 (\mu, \delta_{\bullet}^{u_n})\bigr)$
converges at an even faster rate than $\bigl( n d_1 (\mu, \delta_{\bullet}^{\bullet , n})\bigr)$.

\begin{example}
For the Cantor distribution $\mu$ and its inverse $\mu^{-1}$ in Example \ref{ex320}, Theorem \ref{thm450} simply yields
$\lim\nolimits_{n\to \infty} n d_{\epsilon} (\mu,
\delta_{\bullet}^{\bullet, n}) =
\lim\nolimits_{n\to \infty} n d_{\epsilon} (\mu^{-1},
\delta_{\bullet}^{\bullet, n}) = 0$. An elementary analysis shows that both sequences $\bigl( n^{1/c} d_{\epsilon} (\mu,
\delta_{\bullet}^{\bullet, n} )\bigr)$ and $\bigl( n^{1/c} d_{\epsilon} (\mu^{-1},
\delta_{\bullet}^{\bullet, n} )\bigr)$ are divergent, yet bounded
above and below by positive constants, where $c = \log2/\log3 <
1$ is the Hausdorff dimension of both the set $C = G_{\mu}$ and the
measure $\mu$. It seems plausible
that Theorem \ref{thm450} can similarly be refined for a wide class of
self-similar (singular) distributions, thus complementing known $d_{\sf
  W}$-quantization results \cite{GL, GL05, K, Poe}.
\end{example}

To establish one other interesting property of best
$d_{\epsilon}$-approximations, recall from Proposition \ref{prop209}
that if $d_{\epsilon} (\mu, \delta_{x_n}^{p_n}) =
\ell_{\bullet}^{\bullet, n}$ then $p_n$ can easily be determined from
$x_n$ (or vice versa). Thus $x_n$ (or $p_n$) alone already captures
$\delta_{x_n}^{p_n}$ to a large extent, and it is natural to
ask, for instance, whether $x_{n,1}, \ldots , x_{n,n}$, i.e., the locations
of best $d_{\epsilon}$-approximations of $\mu \in \cP$ conform to an
{\em asymptotic point distribution\/} as $n\to \infty$, referred to as
the {\em point density measure\/} of $\mu$ in \cite{GL05}. In the context
of best $d_{\sf W}$-approximations (or -quantizations), and under the
appropriate assumptions, this question
has a positive answer; see, e.g., the ``empirical measure
theorem'' \cite[Thm.7.5]{GL} and variants thereof \cite{GL05}. As is
the case with Theorem \ref{thm450} and Proposition \ref{prop454}, the
result for best $d_{\epsilon}$-approximations again is simpler than its
$d_{\sf W}$-counterpart in that the asymptotic point distribution
is readily identified whenever $\mu \in \cP$ is non-singular, and no
further assumptions on $\mu$ are needed. In fact, it even is possible
to allow for slightly more general $x_n$. To concisely state the
result, for every $\mu \in \cP$ with $\mu_{\sf A}\ne 0$, define
$\mu_{\epsilon}^{\ast}\in \cP$ via
$$
\frac{{\rm d}\mu_{\epsilon}^{\ast}}{{\rm d}\lambda} =
\frac{\displaystyle  \Omega \! \left( \frac1{\epsilon
   } \cdot \frac{{\rm d} \mu_{\sf A}}{{\rm d}
     \lambda}\right)}{\displaystyle \int_{\R} \Omega \! \left( \frac1{\epsilon
   } \cdot \frac{{\rm d} \mu_{\sf A}}{{\rm d} \lambda}\right) {\rm
   d}\lambda } \quad \forall \epsilon > 0 \, . 
$$
Clearly, $\mu_{\epsilon}^{\ast}$ is absolutely continuous, and
$\mu_{\epsilon}^{\ast} = \mu$ for some (in fact, all) $\epsilon > 0$
if and only if $\mu$ is uniform, i.e., $\mu = \lambda (\, \cdot \,
\cap B)/\lambda (B)$ for some Borel set $B$ with $\lambda (B) \in
\R^+$. Also, given any $\mu \in \cP \setminus \cP_{\infty}^*$, i.e.,
$\# \, \mbox{\rm supp}\, \mu = \infty$, call a sequence $(x_n)$, with
$x_n \in \Xi_n$ for every $n\in \N$, {\em asymptotically\/} $d_{\epsilon}${\em
  -minimal for\/} $\mu$ if
$$
\lim\nolimits_{n\to \infty} \frac{d_{\epsilon} (\mu,
  \delta_{x_n}^{\bullet})}{d_{\epsilon} (\mu,
  \delta_{\bullet}^{\bullet, n})} = 1 \, .
$$
Thus, for instance, $(x_n)$ is asymptotically $d_{\epsilon}$-minimal for
$\mu \in \cP \setminus \cP_{\infty}^*$ whenever
$\delta_{x_n}^{p_n}$, with $x_n \in \Xi_n$, $p_n \in \Pi_n$, is a best $d_{\epsilon}$-approximation of $\mu$ for every $n\in \N$.

\begin{theorem}\label{thm470}
Let $\mu\in \cP$ and $\epsilon > 0$. If $\mu_{\sf A} \ne 0$ and
$(x_n)$ is asymptotically $d_{\epsilon}$-minimal for $\mu$, then
\begin{equation}\label{eq471}
\lim\nolimits_{n\to \infty} \frac{\# \{ 1\le j \le n : x_{n,j} \in
  I\}}{n} = \mu_{\epsilon}^{\ast} (I) \quad
\forall I \subset \R, \, I \: \mbox{\it an interval} \, .
\end{equation}
\end{theorem}

\begin{proof}
For convenience, write $f = F_{\mu} = f_{\sf A} + f_{\sf S}$ with
$f_{\sf A}, f_{\sf S}\in \cF$ such that $\lambda_{f_{\sf A}} =
\mu_{\sf A}$ and $\lambda_{f_{\sf S}} = \mu_{\sf S}$. (The functions
$f_{\sf A}, f_{\sf S}$ can be made unique, for instance,  by requiring
that $f_{\sf A} (-\infty) = f_{\sf S} (-\infty) = 0$.) Also, let $G=G_{\mu}$,
$\ell_n = d_{\epsilon} (\mu, \delta_{x_n}^{\bullet})$ for all
$n\in \N$, and define $G^{\dagger}$ as in (\ref{eq4503Z}). Once again
it suffices to consider the case of $\epsilon = 1$. Note that
$\mu_{\sf A} \ne 0$ implies $\ell_n > 0$ for every $n$, and
$\lim_{n\to \infty} n \ell_n = \int_{\R} \Omega ({\rm d} \mu_{\sf A}/
{\rm d} \lambda) >0$, by Proposition \ref{prop454} and the assumed asymptotic
$d_{\epsilon}$-minimality of $(x_n)$.

Fix a non-empty interval $I = \: ]y,z]$ with $y,z\in \R$. Perturbing
$x_n$ slightly if necessary, without altering $\# \{ 1\le j \le n :
x_{n,j} \in I\}$ or increasing $d_1 (\mu, \delta_{x_n}^{\bullet})$ for any $n$, it
may be assumed that $x_{n,j}< x_{n,j+1}$ for all $n\in \N$ and $j=0,
\ldots , n$. Thus for every $x\in \R$ there exists a unique $j_n(x)
\in \{0, \ldots , n\}$ with $x_{n, j_n(x)} \le x < x_{n, j_n(x) +
  1}$. By Proposition \ref{prop209},
$$
f_{\sf A} (x_{n,j+1} - \ell_n) + f_{{\sf S}-} (x_{n,j+1} - \ell_n) -
\ell_n \le f_{\sf A} (x_{n,j} + \ell_n) + f_{\sf S} (x_{n,j} + \ell_n)
+\ell_n \quad \forall j = 0, \ldots , n \, ,
$$
and consequently also
\begin{equation}\label{eq472}
f_{\sf A} (x_{n,j+1} - \ell_n) - f_{\sf A} (x_{n,j} + \ell_n) \le 2
\ell_n \quad \forall j =0, \ldots , n \, . 
\end{equation}
Fix any $a>0$, and recalling that $f_{\sf A}$ is differentiable
$\lambda$-almost everywhere with $f_{\sf A}'\ge 0$ and $0< \int_{\R}
f_{\sf A}' \, {\rm d}\lambda = \mu_{\sf A} (\R) \le 1$, pick a
continuous function $g: \R \to \R^+$ with $\int_{\R} |f_{\sf A}' -
g|\, {\rm d}\lambda < a$.

Let $K_n = \{ 0 \le j \le n : x_{n,j} + \ell_n < x_{n,j+1} - \ell_n\}$
which may not be all of the set $\{0, \ldots , n \}$ but does contain
$0,n$ in any case. On the one hand, if $j \in K_n \setminus \{0,n\}$
let $J_{n,j} = [x_{n,j} + \ell_n , x_{n,j+1} - \ell_n]$ and
$\lambda_{n,j} = \lambda (J_{n,j}) = x_{n,j+1} - x_{n,j} - 2\ell_n >
0$, and deduce from (\ref{eq472}) that
$$
\frac1{\lambda_{n,j}} \int_{J_{n,j}} g \, {\rm d}\lambda \le
\frac{2\ell_n - \displaystyle \int_{J_{n,j}} (f_{\sf A}' - g) \, {\rm
    d}\lambda}{ x_{n,j+1} -x_{n,j} - 2\ell_n} \, ,
$$
and consequently
\begin{equation}\label{eq473}
\ell_n \ge \Omega \left(  \frac1{\lambda_{n,j}} \int_{J_{n,j}} g \,
  {\rm d}\lambda  \right) (x_{n,j+1} - x_{n,j}) - \frac12
\int_{J_{n,j}} |f_{\sf A}' - g| \, {\rm d}\lambda \, ;
\end{equation}
with the usual convention $0 \cdot (\pm \infty) = 0$, (\ref{eq473}) is
correct also for $j=0,n$. On the other hand, if $j\not \in K_n$ then
clearly $\ell_n \ge \frac12 (x_{n,j+1} - x_{n,j})$. With
(\ref{eq473}) as well as the definitions of $j_n(x)$ and $K_n$,
therefore,
\begin{align}\label{eq474}
\bigl( 1 + j_n (z) - j_n (y)\bigr) \ell_n & \ge  \sum\nolimits_{j\in \{
  j_n(y), \ldots , j_n(z)\} \cap K_n} \Omega \left(  \frac1{\lambda_{n,j}} \int_{J_{n,j}} g \,
  {\rm d}\lambda  \right) (x_{n,j+1} - x_{n,j})  \nonumber \\
& \qquad - \frac12  \sum\nolimits_{j\in \{
  j_n(y), \ldots , j_n(z)\} \cap K_n}  \int_{J_{n,j}} |f_{\sf A}' - g|
\, {\rm d}\lambda \nonumber \\
& \qquad + \sum\nolimits_{j\in \{
  j_n(y), \ldots , j_n(z)\} \setminus K_n} \frac12 (x_{n,j+1} -
x_{n,j}) \nonumber \\[2mm]
& \ge \int_{[x_{n,j_n(y)}, x_{n,j_n(z) +1}]} h_n \, {\rm d}\lambda - \frac12
\int_{[x_{n,j_n(y)}, x_{n,j_n(z) +1}]} |f_{\sf A}' - g| \, {\rm
  d}\lambda \nonumber \\
& \ge \int_I h_n \, {\rm d}\lambda -\frac{a}{2} \, ,
\end{align}
where the piecewise constant function $h_n : \R \to \R^+$ is given by
$$
h_n (x) = \left\{
\begin{array}{ll}
\displaystyle \Omega \! \left(  \frac1{\lambda_{n,j_n(x)}} \int_{J_{n,j_n(x)}} g \,
  {\rm d}\lambda  \right) & \mbox{\rm if } j_n(x) \in K_n\, , \\[2mm]
\displaystyle \frac12 & \mbox{\rm if } j_n(x) \not \in K_n \, .
\end{array}
\right.
$$
If $j_n(x)\not \in K_n$ for all sufficiently large $n$ then clearly
$\lim_{n\to \infty} h_n(x) = \frac12$, whereas if $x\in G \setminus
G^{\dagger}$ and $j_n (x) \in K_n$ for infinitely many $n$ then
$\liminf_{n\to \infty} h_n(x) \ge \Omega \bigl( g(x)\bigr)$ because,
similarly to (\ref{eq4503A}),
$$
\lim\nolimits_{n\to \infty} [x_{n, j_n(x)}, x_{n,j_n(x) +1}] = \{x\}
\quad \forall x \in G \setminus G^{\dagger} \, .
$$
In summary, therefore,
\begin{equation}\label{eq475}
\liminf\nolimits_{n\to \infty} h_n(x) \ge \Omega \bigl( g(x) {\bf
  1}_G(x)\bigr) \quad \mbox{\rm for $\lambda$-almost every } x\in \R
\, .
\end{equation}
Note that $j_n(z) - j_n(y) = \# \{ 1\le j \le n : x_{n,j} \in
I\}$. Consequently, (\ref{eq474}), Proposition \ref{prop454} with
$\int_{\R} \Omega (f_{\sf A}') \, {\rm d}\lambda > 0$, and Fatou's
lemma applied to (\ref{eq475}), together yield
$$
\liminf\nolimits_{n\to \infty} \frac{\# \{ 1\le j \le n : x_{n,j} \in
  I\}}{n} \ge \frac{\displaystyle \int_I \Omega (g {\bf 1}_G) \, {\rm
    d}\lambda - \frac{a}{2}}{\displaystyle \int_{\R} \Omega (f_{\sf A}') \, {\rm
    d}\lambda } \, .
$$
Recall that $f_{\sf A}' = 0$ on $\R \setminus G$, hence
\begin{align*}
\int_I \Omega (g {\bf 1}_G) \, {\rm d}\lambda & = \int_I \Omega (f_{\sf
  A} ' ) \, {\rm d}\lambda + \int_{I \cap G} \bigl( \Omega (g {\bf
  1}_G) - \Omega (f_{\sf A}') \bigr)  \, {\rm d}\lambda \\
& \ge \int_I \Omega (f_{\sf A} ' ) \, {\rm d}\lambda  - \frac12 \int_I
|f_{\sf A}' - g | \, {\rm d}\lambda  \ge \int_I \Omega (f_{\sf A} ' )
\, {\rm d}\lambda  - \frac{a}{2} \, ,
\end{align*}
and consequently
$$
\liminf\nolimits_{n\to \infty} \frac{\# \{ 1\le j \le n : x_{n,j} \in
  I\}}{n} \ge \mu_{1}^{\ast} (I) - \frac{a}{ \displaystyle \int_{\R} \Omega
  (f_{\sf A} ' ) \, {\rm d}\lambda   } \, .
$$
Since the number $a>0$ as well as the interval $I\subset \R$ have been arbitrary,
and since $\mu_{1}^{\ast} (\R) = 1$,
$$
\lim\nolimits_{n\to \infty} \frac{\# \{ 1\le j \le n : x_{n,j} \in
  I\}}{n} =  \mu_1^{\ast} (I) \quad \forall I \subset \R, I \:
\mbox{\rm an interval}\, ,
$$
i.e., (\ref{eq471}) holds as claimed.
\end{proof}

Note that Theorem \ref{thm470} in particular asserts that if $\mu\in \cP$ is
non-singular and $(\delta_{x_n}^{p_n})$, with $x_n \in \Xi_n$ and $p_n
\in \Pi_n$ for every $n\in \N$, is {\em any\/} sequence of
best $d_{\epsilon}$-approximations of $\mu$, then the sequence
$(\delta_{x_n}^{u_n})$, obtained by ``forgetting'' the optimal
weights and instead assigning equal weight $1/n$ to each atom, converges
weakly to $\mu_{\epsilon}^{\ast}$. It seems rather remarkable that $(\delta_{x_n}^{u_n})$ always
converges, and to a limit that is independent of $(x_n)$. By contrast,
simple examples show that $(\delta_{x_n}^{u_n})$ may diverge if $\mu$
is singular; cf.\ \cite{GL05}.

\begin{example}\label{ex48om}
Let $\mu= {\sf exp}(a)$ with $a>0$. With $\ell_{\bullet}^{\bullet, n}$
and the (unique) best $d_{\epsilon}$-approximation
$\delta_{x_n}^{p_n}$  of $\mu$ found in Example \ref{ex212}, it is
readily confirmed that for any $n\in \N$ and $x\in \R^+$ the number
$\# \{ 1 \le j \le n: x_{n,j} \le x\}$ equals the largest integer not larger than
$$
n - \frac{\epsilon}{2 a \ell_{\bullet}^{\bullet, n}} \log \left( 
1 +  \left(\frac{e^{- a (x +  \ell_{\bullet}^{\bullet, n}/\epsilon)}
  }{ \ell_{\bullet}^{\bullet, n}} - 1 \right) \tanh \frac{a
  \ell_{\bullet}^{\bullet, n} }{\epsilon} \right) \, .
$$
From this, a straightforward calculation utilizing (\ref{eq2155})
yields
$$
\lim\nolimits_{n\to \infty} \frac{\# \{ 1\le j \le n : x_{n,j} \le
  x\}}{n} = 1 - \frac{\log (1 + a e^{-ax}/\epsilon)}{\log (1+
  a/\epsilon)} \quad \forall x \in \R^+ \, .
$$
Thus the asymptotic point density of $(x_n)$ is
$$
\frac{a^2}{\log (1+a/\epsilon)} \cdot \frac1{a + \epsilon e^{ax}} =
\frac{\Omega (ae^{-ax}/\epsilon)}{\int_{\R} \Omega (a
  e^{-ay}/\epsilon) \, {\rm d}y} = \frac{{\rm
    d}\mu_{\epsilon}^{\ast}}{{\rm d}\lambda} (x)  \quad \forall x \in \R^+ \, ,
$$
in perfect agreement with Theorem \ref{thm470}. Note that unlike for
best $d_{\sf W}$-approximations \cite[Thm.7.5]{GL}, this asymptotic
point distribution is not exponential; see also Figure \ref{fig3}.

For another simple example, let $\mu$ be a normal distribution with
mean $0$ and variance $\sigma^2 > 0$. While no explicit formula is
available for the (unique) best $d_{\epsilon}$-approximation
$\delta_{x_n}^{p_n}$ of $\mu$ in general, Theorem \ref{thm470} yields
$$
\frac{{\rm d}\mu_{\epsilon}^{\ast}}{{\rm d}\lambda} (x) =
- \frac1{\displaystyle {\rm Li}_{1/2} \!\left( - \frac1{\epsilon \sqrt{2\pi
      \sigma^2}} \right)} \cdot \frac1{
\sqrt{2\pi \sigma^2}  + 2\pi \epsilon \sigma^2 
e^{x^2/(2\sigma^2) } } \quad \forall x \in \R \, ,
$$
as the asymptotic point density of $(x_n)$. Again, this asymptotic
point distribution is not normal, unlike its $d_{\sf W}$-counterpart.
\end{example}

\begin{figure}[ht]
\psfrag{muexp}[l]{\small $\mu = {\sf exp}(1)$}
\psfrag{munorm}[l]{\small $\mu =N(0,1)$}
\psfrag{tf}[l]{\small $\displaystyle \frac{{\rm d\mu}}{{\rm d}\lambda} = e^{-x}$}
\psfrag{tfast}[l]{$\displaystyle \frac{{\rm d\mu_1^{\ast}}}{{\rm
      d}\lambda} =\displaystyle
  \frac1{ (1+e^x)\log 2}$}
\psfrag{tg}[r]{\small $\displaystyle \frac{{\rm d\mu}}{{\rm d}\lambda} =\frac{e^{-x^2/2}}{\sqrt{2\pi}}$}
\psfrag{tgast}[l]{\small $\displaystyle \frac{{\rm d\mu_1^{\ast}}}{{\rm
      d}\lambda} =\displaystyle - \frac{1/
    {\rm Li}_{1/2} (-\frac1{\sqrt{2\pi}}) }{\sqrt{2\pi}+ 2\pi e^{x^2/2}}$}
\psfrag{th0}[]{\small $0$}
\psfrag{th1}[]{\small $1$}
\psfrag{th2}[]{\small $2$}
\psfrag{thm2}[]{\small $-2$}
\psfrag{th3}[]{\small $3$}
\psfrag{tv1}[]{\small $1$}
\psfrag{tv02}[]{\small $0.2$}
%
%
\begin{center}
\includegraphics{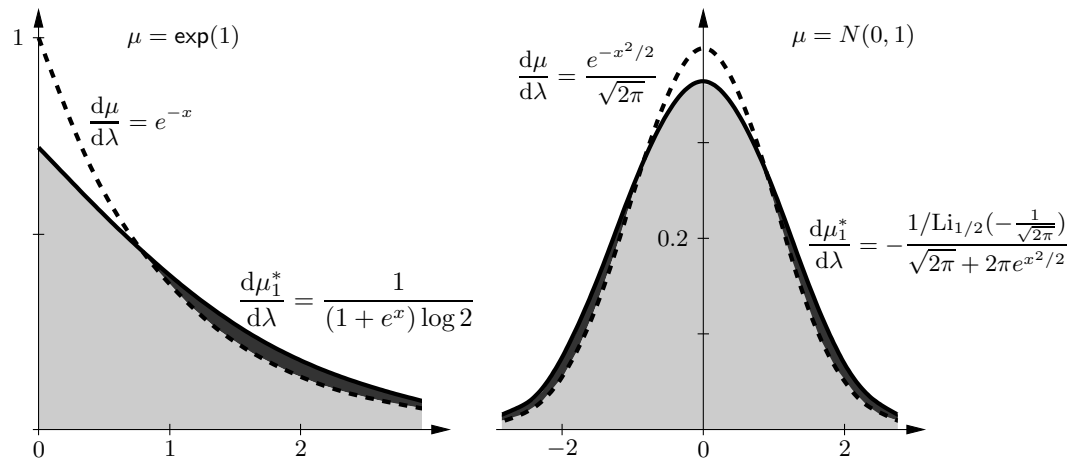}
\end{center}
\caption{Comparing the standard exponential (left) and standard normal
densities (broken curves) to the asymptotic point densities of their
respective best $d_1$-approximations (solid curves); see Example \ref{ex48om}.}\label{fig3}
\end{figure}

\subsubsection*{Acknowledgements}
The first author was partially supported by an {\sc Nserc} Discovery
Grant. He is grateful to T.\ Hill for helpful comments, and to R.\
Zweim\"uller for pointing out the classical reference \cite{DE}.

\end{document}